\documentclass[11pt]{amsart}

\usepackage[a4paper,hmargin=3.5cm,vmargin=3.5cm]{geometry}
\usepackage{amsfonts,amssymb,amscd,amstext}
\usepackage{graphicx}
\usepackage[dvips]{epsfig}
\usepackage{quoting}
\usepackage{hyperref}

\usepackage{fancyhdr}
\input xy
\xyoption{all}

\usepackage{enumerate}
\usepackage{titlesec}
\usepackage{mathrsfs}

\titleformat{\section}
{\filcenter\bfseries} {\thesection{.}}{0.2cm}{}
\titleformat{\subsection}[runin]
{\bfseries} {\thesubsection{.}}{0.15cm}{}[.]
\titleformat{\subsubsection}[runin]
{\em}{\thesubsubsection{.}}{0.15cm}{}[.]

\usepackage[up,bf]{caption}

\newtheorem{theorem}{Theorem}[section]
\newtheorem{proposition}[theorem]{Proposition}
\newtheorem{lemma}[theorem]{Lemma}
\newtheorem{claim}[theorem]{Claim}
\newtheorem{corollary}[theorem]{Corollary}

\theoremstyle{definition}
\newtheorem{definition}[theorem]{Definition}
\newtheorem{remark}[theorem]{Remark}

\numberwithin{equation}{section}
\numberwithin{figure}{section}

\newcommand\Cscr{\mathscr{C}}

\newcommand\Oscr{\mathscr{O}}

\newcommand\C{\mathbb{C}}
\newcommand\D{\overline{\mathbb D}}

\renewcommand\D{\mathbb D}

\newcommand\N{\mathbb{N}}
\renewcommand\P{\mathbb{P}}

\newcommand\R{\mathbb{R}}

\newcommand\Z{\mathbb{Z}}

\renewcommand\c{\mathbb{C}}

\newcommand\wt{\widetilde}
\newcommand\wh{\widehat}

\renewcommand\span{\mathrm{span}}

\newcommand\supp{\mathrm{supp}}

\def\span{\mathrm{span}}

\usepackage{color}
\usepackage[color=blue!20]{todonotes}

\begin{document}


\fancyhead[LO]{Mittag-Leffler theorem for proper minimal surfaces}
\fancyhead[RE]{Antonio Alarc\'on and Tja\v{s}a Vrhovnik}
\fancyhead[RO,LE]{\thepage}

\thispagestyle{empty}


\begin{center}

{\bf\Large The Mittag-Leffler theorem for proper minimal surfaces and directed meromorphic curves}

\medskip

%
%
{\bf Antonio Alarc\'on and  Tja\v{s}a Vrhovnik}
\end{center}


%
%
\medskip

\begin{quoting}[leftmargin={5mm}]
{\small
\noindent {\bf Abstract}\hspace*{0.1cm}
We establish a Mittag-Leffler-type theorem with approximation and interpolation for meromorphic curves $M\to \C^n$ ($n\geq 3$) directed by Oka cones in $\C^n$ on any open Riemann surface $M$. We derive a result of the same type for proper conformal minimal immersions $M\to \R^n$.  This includes interpolation in the poles and approximation by embeddings, the latter if $n\ge 5$ in the case of minimal surfaces.
As applications, we show that complete minimal ends of finite total curvature in $\R^5$ are generically embedded, and characterize those open Riemann surfaces which are the complex structure of a proper minimal surface in $\R^3$ of weak finite total curvature. 

\noindent{\bf Keywords}\hspace*{0.1cm} 
minimal surface,
Mittag-Leffler theorem,
meromorphic immersion,
Riemann surface,
directed immersion,
Oka manifold. 


\noindent{\bf Mathematics Subject Classification (2020)}\hspace*{0.1cm} 
32E30, 
32H04, 
53A10, 
53C42. 
}
\end{quoting}



\section{Introduction 
}
\label{sec:intro}

\noindent
The last two decades have witnessed the development of the approximation and interpolation theories for minimal surfaces in Euclidean space $\R^n$, $n\ge 3$ (see the monograph by Alarc\'on, Forstneri\v c, and L\'opez~\cite{alarcon2021minimal}). In particular, analogous results to the seminal Runge, Weierstrass, and Mittag-Leffler theorems in complex analysis from the 19th century have been recently established for the family of conformal minimal immersions $M\to\R^n$ of any open Riemann surface $M$. This has led to a considerable number of applications to the global theory of minimal surfaces, a main focus of interest since the works of Osserman in the second half of the 20th century. 
The key tool for this task has been the well known close connection of minimal surfaces with complex analysis via the Weierstrass representation formula (see e.g.~\cite[Theorem~2.3.4]{alarcon2021minimal}), together with the classical theories of approximation and interpolation for holomorphic functions on open Riemann surfaces, as well as the modern Oka theory (see~\cite{fornaess2020,forstneric2017stein} for background). Recall that an immersion $u=(u_1,\ldots,u_n):M\to\R^n$ is conformal if and only if the $(1,0)$-differential $\partial u=(\partial u_1,\ldots,\partial u_n)$ of $u$ satisfies the nullity condition
\[
	(\partial u_1)^2+\cdots+(\partial u_n)^2=0.
\]
Such an immersion $u$ is minimal if and only if it is harmonic if and only if $\partial u$ is holomorphic. Conversely, a nowhere vanishing $\C^n$-valued holomorphic 1-form $\phi=(\phi_1,\ldots,\phi_n)$ satisfying $\phi_1^2+\cdots+\phi_n^2=0$ and the period vanishing condition 
\[
	\Re\int_C\phi=0\quad \text{for every closed curve $C$ in $M$}
\] 
integrates to a conformal minimal immersion $u=\Re\int\phi:M\to\R^n$ with $\partial u=\phi$. 
We refer a reader who may wish to delve into the background on minimal surface theory to the classical surveys by Osserman~\cite{osserman1986survey} and Nitsche~\cite{nitsche1989lectures}, as well as to the modern monographs by Colding and Minicozzi~\cite{colding2011course,colding1999minimal}, Dierkes, Hildebrandt, and Sauvigny~\cite{dierkes2010minimal}, and Meeks and P\'erez~\cite{meeks2011classical,meeks2012survey}, among many others; see~\cite{alarcon2021minimal} for a comprehensive study of minimal surfaces from a complex analytic viewpoint.

The Mittag-Leffler theorem from 1884~\cite{mittag1884demonstration} states that for any closed discrete subset $E \subset \C$ and any meromorphic function $f$ on a neighbourhood of $E$ there exists a meromorphic function $g$ on $\C$ which is holomorphic on $\C \setminus E$ and satisfies that the difference $g-f$ is holomorphic at every point in $E$. This result was extended by Florack in 1948~\cite{florack1948regulare} (see also Behnke and Stein~\cite{behnke1949entwicklung}) to functions on open Riemann surfaces, including interpolation on a closed discrete set; approximation on Runge compact subsets was added by Bishop in 1958~\cite{bishop1958subalgebras}. The natural counterpart of a pole of a meromorphic function in minimal surface theory is a complete end of finite total curvature (see e.g.~\cite[\textsection 9]{osserman1986survey} or~\cite[\textsection 4]{alarcon2021minimal}). Indeed, if $u:\overline\D_*=\{\zeta\in\c: 0<|\zeta|\le1\}\to\R^n$ is a complete conformal minimal immersion of finite total curvature then $\partial u$ is holomorphic on $\overline\D_*$ and extends meromorphically to $\overline\D=\overline\D_*\cup\{0\}$ with an effective pole at $0$. Moreover, $u:\overline\D_*\to\R^n$ is a proper map.
Recall that such a $u$ is of \emph{finite total curvature} if 
\[
	\textrm{TC}(u) = \int_{\D_*}KdA > -\infty,
\]
where $g=u^{*}ds^2$ denotes the induced Riemannian metric on $\D_*$, $K\leq 0$ is the Gaussian curvature function, and $dA$ is the area element of $(\D_*,g)$. We refer to~\cite{ChernOsserman1967JAM,osserman1986survey,barbosa1986minimal,yang1994complete} for background on the theory of complete minimal surfaces of finite total curvature.

Alarc\'on and L\'opez proved in~\cite{alarcon2022algebraic} that for any closed discrete subset $E$ of an open Riemann surface $M$ and any complete conformal minimal immersion $u:V\setminus E\to\R^n$ on a neighbourhood $V$ of $E$ such that every point in $E$ is an end of finite total curvature of $u$, there is a complete conformal minimal immersion $\tilde u:M\setminus E\to\R^n$ such that every point in $E$ is an end of finite total curvature of $\tilde u$ and the map $\tilde u-u$ is continuous (hence harmonic) on $V$ and vanishes at all points in $E$; that is, a Mittag-Leffler theorem for complete immersed minimal surfaces. Their theorem also includes approximation and jet-interpolation.
In this paper we shall extend this result in two different directions, namely, we shall establish Mittag-Leffler-type theorems for proper minimal surfaces in $\R^n$, embedded if $n\ge 5$ (see Sec.\ \ref{ss:1.1}), and for meromorphic curves  in $\C^n$ directed by Oka cones (see Sec.\ \ref{ss:1.2}). 

\subsection{The Mittag-Leffler theorem for proper minimal surfaces}\label{ss:1.1}
Here is our first main result.
%
%
\begin{theorem}
\label{th:intro-ML-CMI}
Let $M$ be an open Riemann surface, $\varnothing\neq E\subset M$ be a closed discrete subset, and $V\subset M$ be a locally connected smoothly bounded closed neighbourhood of $E$ all whose connected components are Runge compact sets. Let $n\ge 3$ be an integer and assume that $u:V\setminus E\to\R^n$ is a complete conformal minimal immersion whose restriction to each component of $V\setminus E$ is of finite total curvature; that is, every point in $E$ is a complete end of finite total curvature of $u$. Then there is a proper full\footnote{A conformal minimal immersion $u:M\to\R^n$ is {\em full} if $f(M)$ is not contained in any hyperplane of $\C^n$, where $f=2\partial u/\theta$ for any nowhere vanishing holomorphic 1-form $\theta$ on $M$ \cite[Definition 2.5.2]{alarcon2021minimal}.} conformal minimal immersion $\tilde u:M\setminus E\to\R^n$ such that $\tilde u-u$ is continuous (hence harmonic) at every point in $E$. Moreover, if $n\ge 5$ then we can choose $\tilde u$ to be injective (hence a proper embedding).

Furthermore, if in addition $u:V\setminus E\to \R^n$ is a proper map, and if we are given a number $\varepsilon>0$, a closed discrete subset $\Lambda\subset M$ such that $\Lambda\subset V\setminus E$ and $u|_\Lambda:\Lambda\to\R^n$ is injective if $n\ge 5$, and a map $r:E\cup\Lambda\to\N=\{1,2,\ldots\}$, then there is an immersion $\tilde u$ as above such that $|\tilde u-u|<\varepsilon$ on $V$ and $\tilde u-u$ vanishes to order $r(p)$ at each point $p\in E\cup\Lambda$.
\end{theorem}

This result may be upgraded to include Mergelyan approximation on admissible sets, as opposed to mere Runge approximation (cf.~Theorem~\ref{th:Mittag-Leffler-properness}), and control on the flux as in~\cite[Theorem~6.1]{alarcon2022algebraic}. Nevertheless, the latter is well understood and we do not repeat it here.
The main improvements of Theorem~\ref{th:intro-ML-CMI} with respect to the result by Alarc\'on and L\'opez in~\cite{alarcon2022algebraic} are that we can ensure the conformal minimal immersion $\tilde u:M\setminus E\to\R^n$ to be {\em proper} (instead of just complete), and in addition an {\em embedding} if $n\ge 5$. Note that a generic conformal minimal immersion $M\to\R^n$ is complete and nonproper~\cite{alarcon2024generic}, and so finding proper examples is a considerably more difficult task.
On the other hand, Theorem~\ref{th:intro-ML-CMI} gives the first general existence result for embedded minimal surfaces in $\R^5$ with complete ends of finite total curvature; see~(ii) in the following immediate corollary. 
%
%
\begin{corollary}\label{co:intro-embedded}
If $M$ is an open Riemann surface and $E\subset M$ is a closed discrete subset, then the following hold.
\begin{enumerate}[\rm (i)]
\item There is a proper full conformal minimal immersion $u:M\setminus E\to\R^3$ such that every point in $E$ is an embedded end of finite total curvature of $u$.
\item There is a proper full conformal minimal embedding $u:M\setminus E\to\R^n$ ($n\geq 5$) such that every point in $E$ is an end of finite total curvature of $u$.
\end{enumerate}
\end{corollary}
Statement~(i) contributes to a better understanding of the Sullivan and Schoen-Yau conjectures concerning the conformal properties of proper minimal surfaces in $\R^3$; see~\cite{SchoenYau1997lectures,morales2003existence,AlarconLopez2012JDG} and also~\cite[\textsection 3.10]{alarcon2021minimal} for a discussion of these problems.
Recall that, by a result of Jorge and Meeks~\cite{JorgeMeeks1983T}, a complete minimal end of finite total curvature $u:\overline\D_*\to\R^3$ is embedded (on a neighbourhood of $0$) if and only if $\partial u$ has a pole at $0$ of order $2$. The special cases of the corollary when $E=\varnothing$ were established in~\cite{AlarconLopez2012JDG} and~\cite{alarcon2016embedded}, respectively.
Delving into the direction of statement~(ii), we record the following corollary to the effect that complete minimal ends of finite total curvature in $\R^n$, $n\ge 5$, are generically embedded.
%
%
\begin{corollary}\label{co:intro-embedded-ends}
Let $K$ be a smoothly bounded compact domain in an open Riemann surface, $\varnothing\neq E\subset\mathring K$ be a finite set, and $u:K\setminus E\to\R^n$, $n\ge 5$, be a complete conformal minimal immersion of finite total curvature. Then for any $\varepsilon>0$ and any $r\in\N$ there is a complete conformal minimal embedding of finite total curvature $\tilde u:K\setminus E\to\R^n$ such that $\tilde u-u$ is continuous on $K$, $|\tilde u-u|<\varepsilon$ on $K$ and $\tilde u-u$ vanishes to order $r$ at every point in $E$.
\end{corollary}

We emphasize that the approximation in this corollary is uniform everywhere on $K\setminus E$.
If the map $u|_\Lambda$ given in Theorem~\ref{th:intro-ML-CMI} is not injective, then the statement remains to hold except that $\tilde u$ cannot be chosen injective. Likewise, if $u:V\setminus E\to\R^n$ is not proper, then the statement holds except that $\tilde u$ cannot be a proper map. In this case we may choose $\tilde u:M\setminus E\to\R^n$ to be an almost proper map, hence a complete immersion.\footnote{A map $f:X\to Y$ between topological spaces is {\em almost proper} if the connected components of $f^{-1}(K)$ are all compact for every compact set $K\subset Y$. If $M$ is a smooth surface, then every almost proper immersion $M\to\R^n$ is complete by completeness of $\R^n$. Proper maps are almost proper.} Almost proper maps are in some sense the best class of maps $M\to\R^n$ that can hit an arbitrary countable set in $\R^n$. Thus, we obtain the following result.
%
%
\begin{corollary}\label{co:intro-dense}
Let $M$ be an open Riemann surface and $\varnothing\neq E\subset M$ be a closed discrete subset. For any integer $n\ge 3$ and any countable set $C\subset \R^n$ there is an almost proper (hence complete) conformal minimal immersion $u:M\setminus E\to\R^n$ such that every point in $E$ is an end of finite total curvature of $u$ and $C\subset u(M\setminus E)$; in particular, there is such an immersion $u$ whose image is everywhere dense: $\overline{u(M\setminus E)}=\R^n$. Furthermore, if $n\ge 5$, then there is an immersion $u$ with these properties which is in addition injective.
\end{corollary}
Corollary~\ref{co:intro-dense} is new even in case $E=\varnothing$. Several existence results for complete densely immersed minimal surfaces in $\R^n$ with arbitrary complex structure were provided by Alarc\'on and Castro-Infantes in~\cite{alarcon2018complete,alarcon2019interpolation}.

By a theorem of Huber~\cite{huber1957on} (see also~Chern-Osserman~\cite{ChernOsserman1967JAM}), if $R$ is a Riemann surface, possibly with nonempty boundary $bR\subset R$, and there is a complete conformal minimal immersion $R\to\R^n$ of finite total curvature, then $R$ is of {\em finite conformal type}, meaning that $R$ is conformally equivalent to a Riemann surface of the form $\wh R\setminus X$, where $\wh R$ is a compact Riemann surface, possibly with nonempty compact boundary $b\wh R\subset \wh R$, and $X\subset \wh R\setminus b\wh R$ is finite. If $M$ is an open Riemann surface, then a complete conformal minimal immersion $u:M\to\R^n$ is of {\em weak finite total curvature}, according to L\'opez~\cite{lopez2014exotic,lopez2014uniform}, if the restriction $u|_R:R\to\R^n$ is of finite total curvature for every region $R \subset M$ of finite conformal type.\footnote{A connected subset $R\subset M$ of a topological surface $M$ is a {\em region} if it is closed and, endowed with the induced topology, a topological surface with possibly nonempty boundary $bR\subset R$.} Note that if $M$ is of finite topology, then $u$ is of weak finite total curvature if and only if it is of finite total curvature, and in this case $M$ is a compact Riemann surface with finitely many punctures. By the aforementioned theorem of Huber, if an open Riemann surface $M$ admits a complete conformal minimal immersion $M\to\R^n$ of weak finite total curvature, then every region $R\subset M$ with compact boundary and finite topology must be of finite conformal type. Theorem~\ref{th:intro-ML-CMI} implies that the converse statement also holds; this is (iii)$\Rightarrow$(ii) in the following characterization result.
%
%
\begin{corollary} \label{co:intro-WFTC}
Let $M$ be an open Riemann surface and $n\ge 3$ be an integer. Then the following are equivalent.
\begin{itemize}
\item[\rm (i)] There exists a proper full conformal minimal immersion $M \to \R^n$ of weak finite total curvature.

\item[\rm (ii)] There exists a complete full conformal minimal immersion $M \to \R^n$ of weak finite total curvature.

\item[\rm (iii)] Every region $R\subset M$ with compact boundary and finite topology is of finite conformal type, that is, every annular end of $M$ is a puncture.
\end{itemize}
Furthermore, if $M$ is of infinite topology and $n\ge 5$, then either of the above conditions is equivalent to the following.
\begin{itemize}
\item[\rm (iv)] There exists a proper full conformal minimal embedding $M \to \R^n$ of weak finite total curvature.
\end{itemize}
\end{corollary}
\begin{proof}
The result is known in the case when $M$ is of finite topology. Indeed, {\rm (i)}$\Leftrightarrow${\rm (ii)}$\Rightarrow${\rm (iii)} follows from the classical theorems of Huber~\cite{huber1957on}, Chern-Osserman~\cite{ChernOsserman1967JAM} and Jorge-Meeks~\cite{JorgeMeeks1983T}, while {\rm (iii)}$\Rightarrow${\rm (ii)} is ensured by a recent result of Alarc\'on and L\'opez~\cite[Theorem~1.2]{alarcon2022algebraic} (see also Pirola~\cite{pirola1998algebraic} for $n=3$).

Now assume that $M$ has infinite topology. The implications {\rm (iv)}$\Rightarrow${\rm (i)} if $n\ge 5$ and {\rm (i)}$\Rightarrow${\rm (ii)} are obvious, while {\rm (ii)}$\Rightarrow${\rm (iii)} is again due to Huber~\cite{huber1957on}.
Let us explain how \rm (iii) implies {\rm (iv)}, if $n\ge 5$, and \rm (i). Assuming \rm (iii), $M$ is biholomorphic to an open Riemann surface of the form $M'\setminus E$, where $M'$ is still an open Riemann surface (perhaps of finite topology) and $E\subset M'$ is a (possibly empty) closed discrete subset, such that every proper region in $M'\setminus E$ with compact boundary and finite topology is a finitely punctured compact Riemann surface with compact boundary whose punctures are interior points and belong to $E$ (they are the only ends of the region). Theorem~\ref{th:intro-ML-CMI} applied to $M'$ and $E$ then furnishes us with a proper full conformal minimal immersion (embedding if $n\geq 5$) $u:M'\setminus E\to \R^n$ 
such that every point in $E$ is a complete end of finite total curvature of $u$.
In these conditions, the restriction of the immersion to any proper region of $M'\setminus E$ as above is complete and of finite total curvature, hence $u$ is of weak finite total curvature.
\end{proof}
It is easily seen that on every open Riemann surface $M'$ there is a closed discrete subset $E$ such that $M=M'\setminus E$ satisfies condition~(iii) in Corollary~\ref{co:intro-WFTC} and hence, $M$ is the complex structure of a proper minimal surface in $\R^3$ of weak finite total curvature and of a properly embedded minimal surface in $\R^5$ of weak finite total curvature. There are several restrictions (even topological ones) on an open Riemann surface $M$ of finite conformal type (thus, conformally equivalent to a finitely punctured compact Riemann surface with empty boundary)  to be the complex structure of a complete embedded minimal surface of finite total curvature in $\R^3$; see e.g.~\cite[\textsection 3]{hoffman1997complete} or~\cite[\textsection 5]{LopezMartin1999}. It is an open question, likely very difficult, whether every such $M$ admits a complete conformal minimal embedding of finite total curvature into $\R^4$ or even into $\R^5$.
\subsection{The Mittag-Leffler theorem for directed meromorphic curves}\label{ss:1.2}
A conformal minimal immersion $M\to\R^n$ $(n\ge 3)$ of an open Riemann surface $M$ is locally on every simply connected domain in $M$ the real part of a holomorphic null curve into $\C^n$. Recall that a holomorphic immersion $F:M\to\C^n$ is {\em null} if it is directed by the null quadric
\begin{equation}\label{eq:null-quadric}
	{\bf A}=\{z=(z_1,\ldots,z_n)\in\C^n: z_1^2+\cdots+z_n^2=0\}\subset\C^n,
\end{equation}
meaning that the complex derivative $F'$ of $F$ with respect to any local holomorphic coordinate on $M$ assumes values in ${\bf A}_*={\bf A}\setminus\{0\}$. Equivalently, for any holomorphic 1-form $\theta$ vanishing nowhere on $M$ (such exists by~\cite{grauert1958analytische,grauert1957approximation}), the range of the holomorphic map $dF/\theta:M\to\C^n$ lies in ${\bf A}_*$. We refer to~\cite[\textsection 2]{alarcon2021minimal} for more details. 
More generally, assuming that $A\subset\C^n$ is a closed irreducible conical complex subvariety of $\C^n$ which is smooth away from $0$, a holomorphic immersion $F:M\to\C^n$ is {\em directed} by $A$, or an {\em $A$-immersion}, if its complex derivative $F'$ with respect to any local holomorphic coordinate on $M$ takes its values in $A_*=A\setminus\{0\}$. 
Every such subvariety $A$ is algebraic and is the common zero set of finitely many homogeneous holomorphic polynomials by Chow's theorem (see e.g.~\cite{chirka1989complex}), and it turns out that $A_*$ is the punctured cone on a connected submanifold $Y$ of $\C\P^{n-1}$.
That being so, holomorphic null curves are a special type of directed holomorphic immersions of open Riemann surfaces into $\C^n$. Directed holomorphic immersions were studied by Alarc\'on and Forstneri\v{c}~\cite{alarcon2014null}, who proved an Oka principle with Runge and Mergelyan approximation for holomorphic $A$-immersions, and that every holomorphic $A$-immersion can be approximated uniformly on compact sets by holomorphic $A$-embeddings under the assumption that  $A_*$ is an Oka manifold, which holds if and only if $Y$ is an Oka manifold~\cite[Proposition 4.5]{alarcon2014null}. 
Interpolation was added by Alarc\'on and Castro-Infantes in~\cite{alarcon2019interpolation}; we refer to~\cite{ForstnericLarusson2019CAG,Castro-InfantesChenoweth2020,AlarconLarusson2023} for further recent results in the subject.
The condition on the subvariety $A_*$ to be an Oka manifold arises naturally since holomorphic maps from Stein manifolds to Oka manifolds satisfy the Runge approximation and the Weierstrass interpolation theorems, and more generally, all forms of the Oka principle, in the absence of topological obstructions (see Section \ref{sec:prelim} for further discussion on the Oka property). These tools shall be extensively used in our proofs. The background on the classical Oka-Grauert theory along with the modern developments in Oka theory are presented in the comprehensive book~\cite{forstneric2017stein} by Forstneri\v{c}, and also in the surveys~\cite{ForstnericLarusson2011survey,forstneric2023recent}. 

The proof of the Mittag-Leffler theorem for conformal minimal surfaces by Alarc\'on and L\'opez in~\cite{alarcon2022algebraic} uses in a strong way the special geometry of the null quadric, as well as fairly technical results from the function theory of open Riemann surfaces. Our approach to the proof of Theorem~\ref{th:intro-ML-CMI} is different, relying mainly on modern Oka theory. This enables us to establish analogues for the more general family of directed holomorphic immersions, thereby extending the results from~\cite{alarcon2022algebraic} in a different direction. Given an open Riemann surface $M$, a closed discrete subset $\varnothing\neq E\subset M$, and a cone $A$ as above, we shall say that a holomorphic $A$-immersion $M\setminus E\to\C^n$ is a {\em meromorphic $A$-immersion} if it extends to $M$ as a meromorphic map with an effective pole at each point in $E$.
Here is our second main result.
%
%
\begin{theorem}
\label{th:intro-ML-DirMer}
Let $A\subset \C^n$ $(n\ge 3)$ be a closed irreducible conical subvariety which is not contained in any hyperplane, is smooth away from $0$, and such that $A_*=A\setminus\{0\}$ is an Oka manifold.
Let $M$ be an open Riemann surface, $\varnothing\neq E\subset M$ be a closed discrete subset, and $V\subset M$ be a locally connected smoothly bounded closed neighbourhood of $E$ all whose connected components are Runge compact sets. Let $F:V\setminus E\to\C^n$ be a holomorphic $A$-immersion extending meromorphically to $V$ with an effective pole at each point in $E$. Then, for any number $\varepsilon>0$, any closed discrete subset $\Lambda\subset M$ with $\Lambda\subset V\setminus E$, and any map $r:E\cup \Lambda\to\N$ there is a holomorphic $A$-immersion $\wt F:M\setminus E\to\C^n$ satisfying the following conditions.
\begin{enumerate}[\rm (i)]
\item $\wt F-F$ is continuous (hence holomorphic) on $V$. In particular, $\wt F$ extends meromorphically to $M$ with an effective pole at each point in $E$. 
\item $|\wt F-F|<\varepsilon$ everywhere on $V$.
\item $\wt F-F$ vanishes to order $r(p)$ at every point $p\in E\cup\Lambda$.
\end{enumerate}

Furthermore, the following hold.
\begin{enumerate}[\rm (i)] 
\item[\rm (iv)] If the map $F|_\Lambda:\Lambda\to\C^n$ is injective, then we can choose $\wt F:M\setminus E\to\C^n$ to be in addition injective.

\item[\rm (v)] If $A\cap\{z_1=1\}$ is an Oka manifold and the coordinate projection $\pi_1 \colon A \to \C$ onto the $z_1$-axis admits a local holomorphic section $\rho_1$ near $z_1=0$ with $\rho_1(0) \neq 0$, then we can choose $\wt F:M\setminus E\to\C^n$ to be in addition an almost proper map, and hence a complete immersion.

\item[\rm (vi)] If $A\cap\{z_i=1\}$ is an Oka manifold, the coordinate projection $\pi_i \colon A \to \C$ onto the $z_i$-axis admits a local holomorphic section $\rho_i$ near $z_i=0$ with $\rho_i(0) \neq 0$ for every $i=1,\ldots,n$, and $F:V\setminus E\to \C^n$ is a proper map, then we can choose $\wt F:M\setminus E\to\C^n$ to be in addition proper.
\end{enumerate}
\end{theorem}
The punctured cone $A_*$ on any connected Oka submanifold $Y$ of $\C\P^{n-1}$ satisfies the assumptions in the theorem, so there are lots of examples of subvarieties $A=A_*\cup\{0\}\subset\C^n$ which the result applies to. The extra assumptions on $A$ in statements {\rm (v)} and {\rm (vi)} were already used in \cite[Theorems~7.7 and 8.1]{alarcon2014null} (see also \cite[Theorem~1.3]{alarcon2019interpolation}) and are technically required  in our proof to ensure the global properties. This is due to the need of a Mittag-Leffler-type theorem for meromorphic A-immersions with fixed components, which we establish in Theorem~\ref{th:Mittag-Leffler-fixed}. We note that the null quadric ${\bf A}$ \eqref{eq:null-quadric} directing null curves and minimal surfaces meets all these requirements (see \cite[Example~7.8]{alarcon2014null} or \cite[p.~570]{alarcon2019interpolation}).

Finally, let us record the following analogues of Corollaries~\ref{co:intro-embedded},~\ref{co:intro-embedded-ends}, and~\ref{co:intro-dense} for directed meromorphic curves. In case $E=\varnothing$ this follows from the results in~\cite{alarcon2014null,alarcon2019interpolation}.
%
%
\begin{corollary}\label{co:intro-A1}
If $A$ and $M$ are as in Theorem~\ref{th:intro-ML-DirMer}, then the following hold. 
\begin{enumerate}[\rm (i)]
\item If $A$ is as in Theorem~\ref{th:intro-ML-DirMer}-{\rm (vi)} and $\varnothing \neq E\subset M$ is closed and discrete, then there is a proper full holomorphic $A$-embedding $M\setminus E\to\C^n$ extending meromorphically to $M$ with an effective pole at each point in $E$. (See Proposition~\ref{co:ML-A-proper} for a more precise statement.)

\item Let $K\subset M$ be a smoothly bounded compact domain, $\varnothing \neq E\subset K$ be a finite set, and $F:K\setminus E\to\C^n$ be a holomorphic $A$-immersion extending meromorphically to $K$ with an effective pole at each point in $E$. Then, for any $\varepsilon>0$ and any $r>0$ there is a holomorphic $A$-embedding $\wt F:K\setminus E\to\C^n$ such that $\wt F-F$ is continuous on $K$, $|\wt F-F|<\varepsilon$ on $K$ and $\wt F-F$ vanishes to order $r$ at every point in $E$. In particular, $\wt F$ extends meromorphically to $M$ with effective poles at all points in $E$.

\item For any closed discrete subset $\varnothing \neq E\subset M$ there is a holomorphic $A$-immersion $F:M\setminus E\to\C^n$ extending meromorphically to $M$ with an effective pole at each point in $E$, such that $F(M\setminus E)$ is dense in $\C^n$. If in addition $A$ is as in Theorem~\ref{th:intro-ML-DirMer}-{\rm (v)}, then $F:M\setminus E\to\C^n$ can be chosen to be almost proper.
\end{enumerate}
\end{corollary}

Oka manifolds are special examples of the recently introduced class of Oka-1 manifolds: complex manifolds which are universal targets of holomorphic maps from open Riemann surfaces; see~\cite{alarcon2025oka1,ForstnericLarusson2025oka1} for a precise definition and more information. It remains an open question whether our results (or those in~\cite{alarcon2014null}) remain to hold true when the directing cone $A_*$ is Oka-1 but not Oka.


\subsection*{Organization of the paper} 
We introduce the notation and the relevant definitions in Section~\ref{sec:prelim}.
In Section~\ref{sec:Mergelyan theorem}, we obtain a semiglobal Mergelyan-type theorem for directed meromorphic immersions with approximation and interpolation (see Theorem~\ref{th:Mergelyan} and Proposition~\ref{prop:semiglob}). This is our main technical result and the key tool to be exploited in the rest of the paper.
On the other hand, we prove in Section~\ref{sec:desingularizing} a desingularization or general position lemma for meromorphic $A$-immersions.
By combining these results in a recursive procedure, we establish Theorem~\ref{th:intro-ML-DirMer} (the Mittag-Leffler theorem for directed meromorphic immersions) in Sections~\ref{sec:Mittag-Leffler} to~\ref{sec:properness}; see the more precise version in Theorem~\ref{th:Mittag-Leffler-properness}. 
We point out that Theorem~\ref{th:intro-ML-CMI} on minimal surfaces is granted by a very minor modification of the proof of (the somehow more general) Theorem~\ref{th:intro-ML-DirMer}; one just has to choose the directing cone to be the null quadric in $\C^n$ and ignore the imaginary parts of the periods in the construction. This connection between minimal surfaces and directed immersions is already well understood (see e.g.~\cite[\textsection 7]{alarcon2019interpolation}) and we shall not include the details.


\section{Notation and preliminaries}\label{sec:prelim}

We write $\N=\{1,2,\ldots\}$ and $\Z_+=\N\cup\{0\}$. 
Let $|\cdot|$ denote the Euclidean norm in $\R^n$ for $n\in\N$. Given a compact topological space $K$ and a continuous map $f:K\to\R^n$, we mean by $\|f\|_K$ the maximum norm of $f$ on $K$. 
Given $f,g\colon K  \to \R^n$, we say that $f$ approximates $g$ on $K$, and write $f \approx g$, if $||f-g||_{K}<\varepsilon$ for some $\varepsilon>0$ which is sufficiently small according to the argument.
Throughout the paper, we assume that surfaces are connected and $n\in\N$, $n\geq 3$.

Let $A \subset \C^n$ be a closed irreducible conical complex subvariety of $\C^n$ which is not contained in any hyperplane and is smooth away from $0$.\footnote{A complex subvariety $A \subset \C^n$ is {\em conical} if $\lambda A=A$ for all $\lambda \in \C\setminus \{0\}$.} We denote by $A_{*}:=A\setminus \{0\}$ the punctured subvariety. So, throughout the paper, we shall assume that $A_*$ is smooth and connected.

\begin{definition} \label{def:mero A-imm}
Given an open Riemann surface $M$ and a closed conical complex subvariety $A \subset \mathbb{C}^n$, a holomorphic map $F: M \to \mathbb{C}^n$ is an {\em $A$-immersion} if its complex derivative $F'$ on $M$ with respect to any local holomorphic coordinate on $M$ assumes values in $A_{*}=A\setminus \{0\}$ (see~\cite[Definition~2.1]{alarcon2014null}). Given a closed discrete subset $E \subset M$, by a {\em meromorphic $A$-immersion} $M\setminus E\to\C^n$ we mean a holomorphic $A$-immersion $M\setminus E\to\C^n$ extending to $M$ as a meromorphic map with effective poles at all points in $E$. We denote by
\[
	\mathscr{I}_A(M|E)	
\]
the space of all meromorphic $A$-immersions $M\setminus E\to \C^n$.
An injective meromorphic (or holomorphic) $A$-immersion is said to be a \emph{meromorphic} (or \emph{holomorphic}) \emph{$A$-embedding}.
\end{definition}

Definition~\ref{def:mero A-imm} extends to the case when $M$ is a compact bordered Riemann surface with smooth boundary $bM \subset M$, $E\subset \mathring M= M\setminus bM$ is a finite set, and the map $F \colon M\setminus E \to \C^n$ is continuously differentiable on $M\setminus E$ and meromorphic in $\mathring M $.

Let $F = (F_1,\dots,F_n) \colon M \to \C^n$ be a meromorphic map which is holomorphic on $M\setminus E$ and has effective poles at all points in $E$. Fix a nowhere vanishing holomorphic $1$-form $\theta$ on $M$ (see~\cite{gunning1967immersion}) and write $dF_j=f_j\theta$, $j=1,\dots,n$. 
Then the map $F$ is a meromorphic $A$-immersion $M\setminus E\to\C^n$ if and only if its derivative $f=dF/\theta=(dF_1/\theta,\dots,dF_n/\theta):M\setminus E\to\C^n$ has the range in $A_{*}$. On the other hand, given a meromorphic map $f=(f_1,\dots,f_n) \colon M \to A_{*}$ with effective poles in $E$, the vectorial $1$-form $f\theta = (f_1\theta,\dots,f_n\theta)$ integrates  to a meromorphic $A$-immersion $F: M\setminus E \to \C^n$ in $\mathscr{I}_A(M|E)$  if and only if $f\theta$ is holomorphic and exact on $M\setminus E$. The corresponding meromorphic $A$-immersion is then obtained by
\begin{equation*}
F(p) = F(p_0) + \int_{p_0}^{p} f\theta, \quad p\in M,
\end{equation*}
where $p_0 \in M\setminus E$ and $F(p_0)\in\C^n$ are arbitrary initial point and condition.

\begin{definition} \label{def:full}
Let $A$ be as in Definition~\ref{def:mero A-imm}, $M$ be either an open Riemann surface or a compact bordered Riemann surface, and $\theta$ be a nowhere vanishing holomorphic $1$-form on $M$. A holomorphic $A$-immersion $F \colon M \to \C^n$ is \emph{full} if, setting  $f=dF/\theta$, the $\C$-linear span of the tangent spaces $T_{f(x)}A, \ x\in M$, equals $\C^n$.
{\em Full} meromorphic $A$-immersions are defined analogously.
\end{definition}

Given complex manifolds $X,\ Y$ and a subset $M \subset X$, we denote by $\mathscr{C}^0(M,Y)$ the space of continuous maps $M\to Y$, and by $\mathscr{O}(M,Y)$ the space of holomorphic maps from some neighbourhood of $M$ in $X$ to $Y$. If $X$ is a Riemann surface, $Y\subset\C^n$ is a submanifold and $E\subset \mathring M$ is a finite subset, then 
\begin{equation}\label{eq:Oinfty}
	\mathscr{O}_{\infty}(M|E,Y)
\end{equation} 
denotes the space of meromorphic maps from some neighbourhood of $M$ in $X$ to $Y$ with effective poles exactly at points in $E$. If we do not specify the polar set of meromorphic maps in the latter space, we shall write $\mathscr{O}_{\infty}(M,Y)$ instead.
Moreover, given a closed subset $K \subset X$, we denote by $\mathscr{A}(K,Y) = \mathscr{C}^0(K,Y) \cap \mathscr{O}(\mathring K,Y)$ the space of all continuous maps $K\to Y$ which are holomorphic in the interior $\mathring K \subset K$ of $K$. Similarly, for a finite subset $E \subset \mathring K$, 
\begin{equation}\label{eq:Ainfty}
	\mathscr{A}_{\infty}(K|E,Y)
\end{equation} 
denotes the space of meromorphic maps $\mathring K \to Y$ with effective poles exactly at points in $E$ that are continuous on $K\setminus E$.
When $Y=\C$, we shall omit writing the latter in the corresponding spaces of functions, so we have the spaces $\mathscr{C}^0(M)$, $\mathscr{O}(M)$, $\mathscr{O}_{\infty}(M|E)$, $\mathscr{A}(K)$ and $\mathscr{A}_{\infty}(K|E)$.

We denote by
\begin{equation*}
\textrm{Div}(X) = \Big\{ \prod_{i=1}^\mathfrak{n} p_{i}^{s_i} \colon p_i \in X, \ s_i \in \Z, \ \mathfrak{n} \in \N \Big\}
\end{equation*}
the free commutative group of finite divisors of a set $X$. (Note that we use multiplicative notation.)
Given two divisors $D_1, D_2 \in \textrm{Div}(X)$, we write $D_1 \geq D_2$ whenever the divisor $D_1D_2^{-1}$ is of the form 
$D_1D_2^{-1} = \prod_{i=1}^\mathfrak{n} p_{i}^{s_i}$ with $p_i \in X$, $s_i \geq 0$, $\mathfrak{n} \in \N$. In such case, we say that $D_1D_2^{-1} = D$ is an \emph{effective divisor}.
Moreover, the set $\supp(D) = \{p_i \colon s_i \neq 0 \}$ is called the \emph{support} of $D= \prod_{i=1}^\mathfrak{n} p_{i}^{s_i} \in \textrm{Div}(X)$.

Assume that $M$ is a Riemann surface, $K\subset M$ is a compact subset and  $f \in \mathscr{O}_{\infty}(K)$, $f\not\equiv 0$, is a meromorphic function on a neighbourhood of $K$. Then $(f) \in \textrm{Div}(K)$ denotes the divisor of $f$ in $K$, that is, the quotient of zeros of $f$ by poles of $f$ (counted with multiplicities).
Furthermore, given an effective divisor $D \in \textrm{Div}(\mathring K)$ we write 
\begin{gather}\label{eq:OD(K)}
\mathscr{O}_{D}(K) := \{ f\in \mathscr{O}(K) \colon (f) \geq D \}, \quad
\mathscr{A}_{D}(K) := \mathscr{A}(K) \cap \mathscr{O}_{D}(K'),
\end{gather}
where $K' \subset \mathring K$ is any compact neighbourhood of $\supp(D)$.

\begin{definition} \label{def: admissible set}
(See~\cite[Definition~1.12.9]{alarcon2021minimal}.)
Given an open Riemann surface $M$, an \emph{admissible set} in $M$ is a compact set $S=K\cup \Gamma$, where $K$ is a finite (or possibly empty) union of pairwise disjoint compact domains with piecewise smooth boundaries in $M$ and $\Gamma = \overline {S\setminus K}$ is a finite (or possibly empty) union of pairwise disjoint smooth Jordan arcs and closed Jordan curves intersecting $K$ only at their endpoints (or not at all) and such that their intersections with the boundary $bK$ of $K$ are transverse.
\end{definition}

\begin{remark}\label{rem:Admissible-A}
If $M$ is an open Riemann surface and $E \subset \mathring S=\mathring K$ is a finite subset of an admissible set $S=K\cup \Gamma\subset M$, then the notion of a meromorphic $A$-immersion (see~Definition~\ref{def:mero A-imm}) extends to maps $F\colon S\setminus E \to \C^n$ of class $\mathscr{C}^1(S\setminus E)$ that satisfy $F|_{\mathring S\setminus E}\in\mathscr{I}_A(\mathring S|E)$ and the complex derivative $F'|_{\Gamma}$ with respect to any local holomorphic coordinate on $\Gamma$ assumes values in $A_*$. As before, we write $F \in \mathscr{I}_A(S|E)$.
\end{remark}

\begin{definition} \label{def: skeleton}
Let $S=K\cup \Gamma \subset M$ be a connected admissible set and $Q\subset S$ be a finite set. A {\em skeleton of $S$ based at $Q$} is a finite collection $\{C_j: j=1,\ldots,l\}$ of smooth oriented Jordan arcs in $S$ such that the following hold.
\begin{enumerate}[\rm ({S}1)]
\item $C=\bigcup_{j=1}^l C_j$ is Runge on a neighbourhood of $S$ and a strong deformation retract of $S$.

\item If $i,j\in\{1,\ldots,l\}$, $i\neq j$, and $C_i\cap C_j\neq\varnothing$, then $C_i\cap C_j=\{p_{i,j}\}$ where $p_{ij}$ is a common endpoint of $C_i$ and $C_j$. 

\item Every point in $Q$ is an endpoint of an arc in $\{C_j: j=1,\ldots,l\}$.
\end{enumerate}
\end{definition}

Given a connected admissible set $S$, one can always find skeletons of $S$ based at any finite set $Q\subset S$ as in the definition. 
For a detailed construction of a family of such curves we refer to the proof of~\cite[Proposition~3.3.2]{alarcon2021minimal}. 
We emphasize that by property $\rm (S1)$, $C$ is Runge in $M$ whenever $S$ is so.
Recall that a compact set $K$ in a Riemann surface $M$ is \emph{holomorphically convex} or a \emph{Runge set} in $M$ if and only if $K$ has no holes in $M$ (see~\cite[Lemma~1.12.3]{alarcon2021minimal}).

Let $M$ be an open Riemann surface, $S\subset M$ be a connected Runge admissible set (that is, $S$ is an admissible set which is in addition a Runge subset of $M$), and $E \subset \mathring S, \ \Lambda \subset S$ be disjoint finite subsets. Also let $A \subset \C^n$ be a closed irreducible conical complex subvariety and $\theta$ be a nowhere vanishing holomorphic $1$-form on $M$.
Assume that $\mathscr{C} = \{C_1, \dots, C_l\}$ is a skeleton of $S$ based at $E \cup \Lambda$ and set $C:=\bigcup_{j=1}^{l}C_j$ (see~Definition~\ref{def: skeleton}).
For any map $f\in \mathscr{C}^{0}(C\setminus E,A_*)$ we define the family of maps
\begin{equation} \label{eq:C^0(C,f)}
\mathscr{C}^{0}(C,f) := \left\{h\in \mathscr{C}^{0}(C\setminus E,A_*)\colon h-f\in \mathscr{C}^{0}(C,\C^{n}) \right\}
\end{equation}
and the \emph{period map}
\begin{eqnarray} \label{eq:period map}
\mathscr{P} & = & \left(\mathscr{P}_1, \dots, \mathscr{P}_l \right) \colon \mathscr{C}^{0}(C,f) \longrightarrow (\C^{n})^{l} \nonumber \\
\mathscr{P}(h) & := & \left(\int_{C_{j}}(h-f)\theta \right)_{j=1,\dots,l}, \quad h\in \mathscr{C}^{0}(C,f).
\end{eqnarray}

As already mentioned in the introduction, a special class of complex manifolds, called Oka manifolds, plays an important role in the results of the present paper. They were firstly introduced by Forstneri\v{c} in \cite{forstneric2009oka}, following the works by Oka, Grauert, and Gromov on the classical Oka-Grauert principle~\cite{gromov1989oka}. 
We say that a complex manifold $X$ is an \emph{Oka manifold} if every holomorphic map $K \to X$ from (a neighbourhood of) any compact convex set $K \subset \C^n$ for any $n\in\N$ can be approximated uniformly on $K$ by entire maps $\C^n \to X$ \cite[Definition~5.4.1]{forstneric2017stein}. The main result in this theory is that maps from a Stein manifold to an Oka manifold satisfy all equivalent forms of the Oka principle \cite[Theorem 5.4.4]{forstneric2017stein}. This includes the {\em Oka property with approximation and interpolation} to the effect that given a compact Runge set $K$ in a Stein manifold $M$ and an Oka manifold $X$, every continuous map $f:M\to X$ which is holomorphic on a neighbourhood of $K$ can be approximated uniformly on $K$, and interpolated to any given finite order on a finite set in $K$, by holomorphic maps $M\to X$ that are homotopic to $f$. This shall be a crucial tool in our method of proof.
  

\section{The Mergelyan theorem for meromorphic $A$-immersions} \label{sec:Mergelyan theorem}

In this section we establish Mergelyan and semiglobal approximation results for meromorphic $A$-immersions; see Theorem~\ref{th:Mergelyan} and Proposition~\ref{prop:semiglob}. The former, which is used to prove the latter under extra assumptions, can be seen as our main new technical tool. Oka theory is not required for this task, so the submanifold $A_*$ in Theorem \ref{th:Mergelyan} need not be Oka. On the other hand, we provide a similar result of semiglobal nature in Proposition \ref{prop:semiglob}, under the extra assumption that $A_*$ is Oka.

\begin{theorem} \label{th:Mergelyan}
Let $M$ be an open Riemann surface, $S=K\cup \Gamma$ be an admissible set in $M$, and $E \subset \mathring S$, $\Lambda \subset S$ be disjoint finite subsets.
Let $A \subset \C^{n}$ $(n\geq 3)$ be a closed irreducible conical subvariety which is not contained in any hyperplane and is smooth away from $0$, and let $F \in \mathscr{I}_A(S|E)$.
Then, for any $\varepsilon>0$ and any $r\in \Z_+$ there exist an open neighbourhood $\widetilde{S} \subset M$ of $S$ and a full meromorphic $A$-immersion $G\in\mathscr{I}_A(\widetilde S|E)$ such that the following are satisfied.
\begin{itemize}
\item[\rm (i)] $G-F$ is continuous on $S$, vanishes on $\Lambda\setminus\mathring S$ and vanishes to order $r$ on $E \cup (\Lambda\cap \mathring S)$.

\item[\rm (ii)] $||G-F||_S<\varepsilon$.
\end{itemize}
\end{theorem}

Recall the notation $\mathscr{I}_A(S|E)$ in Remark~\ref{rem:Admissible-A} and Definition~\ref{def:mero A-imm}.
The following result extends \cite[Theorem~7.2]{alarcon2014null} to the meromorphic framework.

\begin{proposition} \label{prop:semiglob}
Let $M$, $S$, $E$, $\Lambda$ and $A\subset \C^n$ ($n\geq 3$) be as in Theorem~\ref{th:Mergelyan} and assume that $S$ is Runge in $M$ and $A_{*}$ is an Oka manifold. 
Let $r\in \Z_+$ and $\varepsilon>0$.
Given a meromorphic $A$-immersion $F \in \mathscr{I}_A(S|E)$, there exists a full meromorphic $A$-immersion $G \in \mathscr{I}_A(M|E)$ satisfying the following conditions.
\begin{itemize}
\item[\rm (i)] The difference $G-F$ extends continuously to $S$, vanishes on $\Lambda \setminus \mathring S$ and vanishes to order $r$ on $E \cup (\Lambda \cap \mathring S)$.

\item[\rm (ii)] $||G-F||_S<\varepsilon$.
\end{itemize}
\end{proposition}

We begin with preparations for the proof of Theorem~\ref{th:Mergelyan}. The following lemma is the key step; a Mergelyan type theorem with interpolation for meromorphic maps into a conical subvariety $A_*\subset\C^n$. Its main feature is the control of the periods of the approximating meromorphic map, which ensures that it indeed integrates to a directed meromorphic immersion.

\begin{lemma} \label{lemma:approx&interp into A, full}
Let $M$, $S$, $E$, $\Lambda$ and $A$ be as in Theorem~\ref{th:Mergelyan} and assume that $S$ is connected. Let $\mathscr{C}=\{C_1,\dots,C_l\}$ be a skeleton of $S$ based at $E\cup \Lambda$ (see Definition~\ref{def: skeleton}).
Then for any $\varepsilon >0$, any $r\in \Z_+$ and any $f\in \mathscr{A}_{\infty}(S|E,A_*)$~\eqref{eq:Ainfty} there exist an open neighbourhood $\widetilde{S} \subset M$ of $S$ and a full map $\tilde{f} \in \mathscr{O}_\infty(\widetilde S|E, A_*)$~\eqref{eq:Oinfty} such that:
\begin{itemize}
\item[\rm (i)] $\tilde{f}-f$ is continuous on $S$ and $||\tilde{f}-f||_S<\varepsilon$.

\item[\rm (ii)] $\mathscr{P}(\tilde{f})=0$, where $\mathscr{P}$ is the period map~\eqref{eq:period map} associated to $\mathscr{C}$, $f$ and a given nowhere vanishing holomorphic 1-form $\theta$ on $M$.

\item[\rm (iii)] $\tilde{f}$ and $f$ agree to order $r$ on $E \cup (\Lambda\cap \mathring{S})$.
\end{itemize}
\end{lemma}

\begin{proof}
The proof consists of adapting the method of sprays developed in~\cite{alarcon2014null} to the meromorphic framework. 
We proceed in two steps. Firstly, we approximate and interpolate $f$ by a full map $\hat{f}\in \mathscr{A}_{\infty}(S|E,A_*)$ with $\mathscr{P}(\hat{f})=0$. Secondly, we find a period dominating spray of maps in $\mathscr{A}_{\infty}(S|E,A_*)$ with core $\hat{f}$ and then approximate and interpolate it by a spray in $\widetilde{S}$. A suitable parameter value will then furnish a map $\tilde{f} \in \mathscr{O}_{\infty}(\widetilde{S}|E,A_*)$ satisfying the conclusion of the lemma.

\textit{Step~1: Approximation and interpolation by a full map.} We say that a map $h\in \mathscr{A}_{\infty}(S|E,A_*)$ is {\em full} if it is full on every relatively open subset $X$ of $S\setminus E$, meaning that the $\C$-linear span of the tangent spaces $T_{h(x)}A$, $x\in X$, equals $\C^n$. 
We shall explain the argument in the case when $S=K$ is a connected compact domain  (hence $\Gamma=\varnothing$). For a general admissible set of the form $S=K\cup \Gamma$ we repeat the argument on every connected component of $K$, however, on each arc in $\Gamma$, a suitable small continuous deformation of $f$ provides fullness on $\Gamma$ while preserving the other conditions; see~\cite[Lemma~7.3]{alarcon2014null}.

Assume that the map $f\in \mathscr{A}_{\infty}(K|E,A_*)$ is not full; otherwise skip the first step.
Denote by $\Sigma(f)$ the $\C$-linear subspace of $\C^{n}$ spanned by all tangent spaces $T_{z}A, \ z\in f(K\setminus E)$. Since $f$ is not full, $\Sigma(f)$ is a proper subspace of $\C^{n}$.

\begin{claim} \label{cl:Step1}
There is a full map $\hat{f}\colon K\setminus E \to A_{*}$ of class $\mathscr{A}_{\infty}(K|E,A_{*})$, such that $\hat{f}-f$ is continuous on $K$, vanishes to order $r$ on $E \cup (\Lambda \cap \mathring K)$, $\hat{f}$ is as close as desired to $f$ on $K$, and $\mathscr{P}(\hat{f})=0$.
\end{claim}

\begin{proof}
Choose points $x_1, \dots, x_k \in K\setminus (E \cup \Lambda)$ such that $T_{f(x_j)}A$,  $j=1,\dots, k$, span $\Sigma(f)$. Set $A':=\{0\} \cup \{z\in A_* \colon T_{z}A \subset \Sigma(f)\}$. Choose a holomorphic vector field $V$ on $A$ tangential to $A$ such that it vanishes at $0$ and is not everywhere tangential to $A'$ along $f(K\setminus E)$. Let $\C\ni t\mapsto \phi(t,z)\in A$ be the flow of $V$ for small values of $|t|$ and $\phi(0,z)=z\in A$.
Fix a nonconstant function $h\in \mathscr{O}(K)$ with $h(x_j)=0, \ j=1,\dots,k$, and vanishing to order $r$ at all points in $E\cup \Lambda$.
For any function $g$ on a small neighbourhood of the constant zero-function in $\mathscr{A}(K)$, define $\Phi(g) \in \mathscr{A}_{\infty} \left(K|E, A_* \right)$ by
\begin{equation*}
	\Phi(g)(p) := \phi \left(g(p)h(p),f(p) \right), \quad p\in K\setminus E,
\end{equation*}
and note that $\Phi(g)\in \mathscr{C}^{0}(C,f)$, where $C=\bigcup_{j=1}^l C_j$; see~\eqref{eq:C^0(C,f)}. Consider the map
\begin{equation*}
g \longmapsto \mathscr{P} \left( \Phi(g) \right) \in (\C^{n})^l,
\end{equation*}
where $\mathscr{P}$ is the period map in Lemma~\ref{lemma:approx&interp into A, full}-\rm(ii).
We have $\mathscr{P}(\Phi(0)) = \mathscr{P}(\phi(0,f)) = \mathscr{P}(f)=0$ and since the vector space $\mathscr{A}(K)$ is of infinite dimension, there exists a nonconstant function $g\in \mathscr{A}(K)$ close to the zero-function on $K$, satisfying $\mathscr{P}(\Phi(g))=0$.
It follows that the map $\hat{f} := \Phi(g) \colon K\setminus E \to A_*$ lies in $\mathscr{A}_{\infty}(K|E,A_{*})$ and satisfies the requirements in Claim~\ref{cl:Step1} except that $\hat f$ may fail to be full. Indeed, note first that if $p\in E\cup \Lambda$, then $(\hat{f}-f)(p) = \Phi(g)(p)-f(p) = \phi(0,f(p))-f(p) = f(p)-f(p) = 0$ (to the correct order, as $h$ vanishes to order $r$ at $p$), and approximation on $K$ follows by construction. 

Let us check that $\dim \Sigma(\hat{f}) > \dim \Sigma(f)$, where $\Sigma(\hat f)$ denotes the $\C$-linear subspace of $\C^{n}$ spanned by all tangent spaces $T_{z}A$ with $z\in \hat f(K\setminus E)$. Firstly, observe that $\Sigma(f)\subset \Sigma(\hat{f})$. Indeed, since $h(x_j)=0, \ j=1,\dots,k$, the inclusion follows from the equality $\hat{f}(x_j)=\Phi(g)(x_j)=\phi(g(x_j)h(x_j), f(x_j)) = \phi(0,f(x_j))=f(x_j)$.
Moreover, $g$ is nonconstant on $K$ (hence so is $gh$) and $V$ is not everywhere tangential to $A'$ along $f(K\setminus E)$, which implies that there exists a point $x_0 \in K \setminus (E\cup \Lambda \cup \{x_1,\dots,x_k\})$ such that $\hat{f}(x_0)\in A\setminus A'$. Therefore, $T_{\hat{f}(x_0)}A \subset \Sigma(\hat{f}) \setminus \Sigma(f)$ and $\dim \Sigma(\hat{f}) > \dim \Sigma(f)$, as claimed. 
Finally, a finite recursive application of this argument enables us to assume that $\dim \Sigma(\hat{f})=n$, thereby concluding that $\hat f$ is full. 
\end{proof}

\textit{Step~2: Approximation and interpolation by maps in $\mathscr{O}_{\infty}(\widetilde{S}|E,A_*)$.}
By the previous step, we may assume that $f$ is full on $S$.
Choose a function $g\in \mathscr{O}(M)$ with
\begin{equation}\label{eq:(g)}
	(g)=\prod_{p\in E}p^{o(p)},
\end{equation}
where $o(p)=\max \{ \textrm{Ord}_{p}(f_i)\colon i=1,\dots,n\}$ and $\textrm{Ord}_{p}(f_i)$ denotes the order of $f_i$ at the pole $p$. 
Define
\begin{equation}\label{eq:tildef0}
	\tilde{f}_0:=fg.
\end{equation} 
Since $f\in \mathscr{A}_{\infty}(S|E,A_{*})$, we have by~\eqref{eq:(g)} that $\tilde{f}_0 \in \mathscr{A}(S,A_{*})$. 
Choose holomorphic vector fields $V_1,\dots,V_N$ on $\C^n$ ($N\in \N$) that are tangential to $A$ along $A$, vanish at $0\in A$ and such that $\span\{V_1(z),\dots,V_N(z)\}=T_{z}A$ for each $z\in A_{*}$. Let $\phi_{t}^{i}$ denote the flow of $V_i$ for small $t\in \C$ and all $i=1,\dots,N$. Recall that $l$ is the number of arcs in the skeleton $\mathscr{C}=\{C_1,\dots,C_l\}$ in the statement of the lemma.

\begin{claim} \label{claim:period spray}
There exist an open neighbourhood $U$ of $0\in \C^{Nl}$ and a map $\widetilde{\Phi}_{\tilde{f}_0}\in \mathscr{A}(U\times S,A_{*})$ such that the map
\begin{gather*}
U\times (S \setminus E) \ni (\zeta,p) \longmapsto \widetilde{\Phi}_{\tilde{f}_0}(\zeta,p)/g(p) \in A_{*}
\end{gather*}
is continuous, holomorphic on $U\times (\mathring S \setminus E)$ and satisfies $\widetilde{\Phi}_{\tilde{f}_0}(0,\cdot)/g=\tilde{f}_0/g=f$. 
Moreover, $\widetilde \Phi_{\tilde f_0}(\zeta,\cdot)/g$ is meromorphic on $\mathring S$,
the difference $\widetilde{\Phi}_{\tilde{f}_0}(\zeta,\cdot)/g-f$ extends continuously to $S$, hence $\widetilde{\Phi}_{\tilde{f}_0}(\zeta,\cdot)/g\in \Cscr^0(C,f)$~\eqref{eq:C^0(C,f)}, $\widetilde{\Phi}_{\tilde{f}_0}(\zeta,\cdot)/g-f$ vanishes to order $r$ on $E \cup (\Lambda\cap \mathring S)$ for each parameter $\zeta \in U$, and the period map $U\ni \zeta \mapsto \mathscr{P}(\widetilde{\Phi}_{\tilde{f}_0}(\zeta,\cdot)/g) \in \C^{nl}$~\eqref{eq:period map} has maximal rank equal to $nl$ at $\zeta=0$.
\end{claim}

\begin{proof}
Let $\gamma_j \subset \mathring C_j$, $j=1,\ldots,l$, be Jordan arcs.
Since $f$ is full on $S$, the same holds for $\tilde{f}_0$~\eqref{eq:tildef0}, implying that the tangent spaces $T_{\tilde{f}_0(x)}A$, $x \in S$, span $\C^n$, which is also true on the arcs $\gamma_j$, $j=1,\dots,l$.
Thus, there exist pairwise distinct points $p_{1,j},\dots,p_{N,j}\in \gamma_j$ such that
\begin{gather} \label{eq:span(V(tilde f0))}
\span \left\{ V_1(\tilde{f}_0(p_{1,j})),\dots, V_N(\tilde{f}_0(p_{N,j})) \right\} = \C^n,\quad j=1,\ldots,l.
\end{gather}
Choose functions $\tilde{f}_{1,j}, \dots, \tilde{f}_{N,j} \in \mathscr{C}^{0}(C,\C)$ with pairwise disjoint supports such that $p_{i,j}$ is an interior point of $\supp(\tilde{f}_{i,j})\subset \gamma_j$ and
\begin{gather} \label{eq:proof-claim-V_i(tilde f0)}
\int_{C_{j}} \frac{\tilde{f}_{i,j}}{g}\cdot(V_i \circ \tilde{f}_0) \theta \approx V_i ( \tilde{f}_0(p_{i,j})), \quad i=1,\dots,N;
\end{gather}
see~\eqref{eq:(g)}.
Set $\widetilde{F}=(\tilde{f}_{1,j},\dots,\tilde{f}_{N,j})_{j=1,\dots,l} \in \mathscr{C}^{0}(C,(\C^N)^l)$ and let $W$ be an open neighbourhood of $0 \in (\C^N)^l$ so small that the map $\Phi_{\widetilde{F}} \colon W \times C \times A_{*} \to A_{*}$,
\begin{multline*} \label{eq:Phi_tildeF}
\Phi_{\widetilde{F}}(\zeta,p,z) := \left(  \phi_{\zeta_{1,1}\tilde{f}_{1,1}(p)}^1 \circ \cdots \circ \phi_{\zeta_{N,1}\tilde{f}_{N,1}(p)}^N \circ \cdots \right. \\
\cdots\left. \circ \phi_{\zeta_{1,l}\tilde{f}_{1,l}(p)}^1 \circ \cdots \circ \phi_{\zeta_{N,l}\tilde{f}_{N,l}(p)}^N \right) (z),
\end{multline*}
is well defined, where $\zeta = \left( (\zeta_{i,j})_{i=1,\dots,N} \right)_{j=1,\dots,l} \in W$.
Also, define the map $\Phi_{\widetilde{F},\tilde{f}_0} \colon W \times C \to A_{*}$ by
\begin{equation*} \label{eq:Phi_tildeFcore}
\Phi_{\widetilde{F},\tilde{f}_0}(\zeta,p) := \Phi_{\widetilde{F}} \big(\zeta,p,\tilde{f}_0(p)\big),
\end{equation*}
that is, a spray of continuous maps with core 
\begin{equation}\label{eq:ftildecore}
	\Phi_{\widetilde{F},\tilde{f}_0}(0,\cdot)=\tilde{f}_0.
\end{equation} 
Observe that $\Phi_{\widetilde{F},\tilde{f}_0}(\zeta,\cdot)/g \in \mathscr{C}^0(C\setminus E,A_{*})$ and $\Phi_{\widetilde{F},\tilde{f}_0}(\zeta,\cdot)/g =\tilde f_0/g= f$ on $C\setminus \bigcup_{i,j}\supp({\tilde{f}_{i,j}})$, hence $\Phi_{\widetilde{F},\tilde{f}_0}(\zeta,\cdot)/g-f=0$ on a neighbourhood of $E \cup \Lambda$ in $C$. Consequently, $\Phi_{\widetilde{F},\tilde{f}_0}(\zeta,\cdot)/g \in \mathscr{C}^0(C,f)$, $\zeta \in W$; see~\eqref{eq:C^0(C,f)}.
We are therefore entitled to define the period map $\mathscr{Q}=(\mathscr{Q}_1,\dots,\mathscr{Q}_l) \colon W \to (\C^n)^l$ by
\begin{equation} \label{eq:sprayQ}
	\mathscr{Q}(\zeta) = \mathscr{P} \big(\Phi_{\widetilde{F},\tilde{f}_0}(\zeta,\cdot)/g \big), \quad \zeta\in W;
\end{equation}
see~\eqref{eq:period map}.
Then, for each pair of indices $i\in\{1,\ldots,N\}$ and $j\in\{1,\ldots,l\}$, we have 
\begin{equation} \label{eq:dQ1}
\frac{\partial\mathscr{Q}_k}{\partial \zeta_{i,j}}\Big|_{\zeta=0} = 0 \quad \text{for }k\in\{1,\ldots,l\},\; k\neq j,
\end{equation}
since $\supp({\tilde{f}_{i,j}})\subset \gamma_j \subset \mathring C_j$, while
\begin{equation} \label{eq:dQ2}
\frac{\partial\mathscr{Q}_j}{\partial \zeta_{i,j}}\Big|_{\zeta=0} = \int_{C_j} \frac{\tilde{f}_{i,j}}{g}\cdot (V_i\circ \tilde{f}_0)\theta  \stackrel{\text{\eqref{eq:proof-claim-V_i(tilde f0)}}}{\approx} V_i(\tilde{f}_0(p_{i,j})).
\end{equation}
The matrix $\Big(\big(\frac{\partial\mathscr{Q}}{\partial\zeta_{i,j}}\big|_{\zeta=0}\big)_{i=1,\dots,N}\Big)_{j=1,\dots,l}$ has block structure with $l\times l$ blocks, each of them being of size $N\times n$, and by~\eqref{eq:dQ1}, these blocks are nonzero only on the diagonal. Therefore, assuming that the approximation in~\eqref{eq:proof-claim-V_i(tilde f0)} is close enough,~\eqref{eq:span(V(tilde f0))},~\eqref{eq:proof-claim-V_i(tilde f0)} and~\eqref{eq:dQ2} guarantee that
\begin{equation*}
\textrm{rank} \left(\Big(\frac{\partial\mathscr{Q}}{\partial\zeta_{i,j}}\Big|_{\zeta=0}\right)_{i=1,\dots,N}\Big)_{j=1,\dots,l} = \sum_{j=1}^{l} \textrm{rank} \Big(\frac{\partial\mathscr{Q}_j}{\partial\zeta_{i,j}}\Big|_{\zeta=0}\Big)_{i=1,\dots,N} = nl.
\end{equation*}
Thus, the implicit function theorem gives a closed ball $U' \subset W\subset (\C^N)^l$ around $0 \in (\C^N)^l$ such that
\begin{equation} \label{eq:Qsubmersion}
\mathscr{Q} \colon U' \to \mathscr{Q}(U') \textrm{ is a holomorphic submersion and } \mathscr{Q}(0) = 0.
\end{equation}
The latter follows from~\eqref{eq:tildef0},~\eqref{eq:ftildecore},~\eqref{eq:sprayQ} and~\eqref{eq:period map}. 
Choose $U'$ so small that
\begin{equation*}
\Phi_{\widetilde{F},\tilde{f}_0}(\zeta,\cdot) \approx \tilde{f}_0 \quad \text{for all }\zeta \in U'.
\end{equation*}
To finish, we shall approximate $\Phi_{\widetilde{F},\tilde{f}_0}$ by a map that satisfies the claim.
For this, choose an integer 
\begin{equation}\label{eq:r0}
	r_0 \geq r + \sum_{p\in E} \sum_{i=1}^{n} \textrm{Ord}_{p}(f_i).
\end{equation} 
Using the Mergelyan theorem, we approximate each $\tilde{f}_{i,j}  \in \mathscr{C}^{0}(C,\C)$ uniformly on $C$ by a function $h_{i,j}\in \mathscr{O}(S)$ with 
\begin{equation} \label{eq: h_ij}
(h_{i,j}) \geq D_1=\prod_{p\in E\cup \Lambda}p^{r_0+1}, 
\end{equation}
which is possible since $\tilde f_{i,j}=0$ on a neighbourhood of $E\cup\Lambda$ and $C$ is Runge in a neighbourhood of $S$. 
Set $H=(h_{1,j},\dots,h_{N,j})_{j=1,\dots,l} \in \left(\mathscr{O}_{D_1}(S)^N \right)^l$; see~\eqref{eq:OD(K)}.
Choose an open neighbourhood $U$ of the origin, $0\in U\subset U'\subset (\C^N)^l$ and consider the map $\Phi_{H,\tilde{f}_0} \colon U\times S \to A_{*}$ given by
\begin{multline*}
\Phi_{H,\tilde{f}_0}(\zeta,p) := \left( \phi_{\zeta_{1,1}h_{1,1}(p)}^1 \circ \cdots \circ \phi_{\zeta_{N,1}h_{N,1}(p)}^N \circ \cdots \right. \\
\left. \cdots\circ \phi_{\zeta_{1,l}h_{1,l}(p)}^1 \circ \cdots \circ \phi_{\zeta_{N,l}h_{N,l}(p)}^N \right) \big( \tilde{f}_0(p) \big),
\end{multline*}
which is a spray of class $\mathscr A(U\times S,A_*)$ with core $\Phi_{H,\tilde{f}_0}(0,\cdot)=\tilde{f}_0$.
It is clear that the map $\Phi_{H,\tilde{f}_0}(\zeta,\cdot)/g$ lies in $\mathscr A_\infty (S|E,A_*)$ for all $\zeta\in U$ and $\Phi_{H,\tilde{f}_0}(0,\cdot)/g = \tilde{f}_0/g = f$. 
Letting $D_2:= \prod_{p\in E\cup \Lambda} p^r$, it follows that $\Phi_{H,\tilde{f}_0}(\zeta,\cdot)/g - f \in \mathscr{A}_{D_2}(S)^n$ for all $\zeta \in U$  (see~\eqref{eq:OD(K)},~\eqref{eq: h_ij}), so, in particular, $\Phi_{H,\tilde{f}_0}(\zeta,\cdot)/g \in \mathscr{C}^0(C,f)$; see~\eqref{eq:C^0(C,f)}.
This implies that the map 
\begin{equation}\label{eq:Qtilde}
\widetilde{\mathscr{Q}}(\zeta) := \mathscr{P}\left( \Phi_{H,\tilde{f}_0}(\zeta,\cdot)/g \right), \quad \zeta \in U,
\end{equation}
is well-defined with $\widetilde{\mathscr{Q}}(0)=\mathscr{P}(f)=0$.
Assuming that each function $h_{i,j}$ is sufficiently close to $\tilde{f}_{i,j}$ on $C$, and provided that $U$ is small enough, we have that the period map $\widetilde{\mathscr{Q}} \colon U \to \C^{nl}$~\eqref{eq:Qtilde} has maximal rank at $\zeta=0$ by~\eqref{eq:Qsubmersion}. It is now clear that 
$\widetilde{\Phi}_{\tilde{f}_0}=\Phi_{H,\tilde{f}_0}$ satisfies the conclusion of the claim.
\end{proof}
%
%
We now continue the proof of Lemma~\ref{lemma:approx&interp into A, full} by proceeding with Step~$2$. We let $\widetilde{S}$ be an open neighbourhood of $S$ in $M$ where the functions $h_{i,j}$ in \eqref{eq: h_ij} are holomorphic and, by the standard Mergelyan theorem, approximate $\tilde{f}_0 \in \mathscr{A}(S,A_{*})$ in \eqref{eq:tildef0} uniformly on $S$ by a function $h_0\in \mathscr{O}(\widetilde{S},A_{*})$ which agrees with $\tilde{f}_0$ to order $r_0+1$ on $(E\cup \Lambda)\cap \mathring S$; see~\eqref{eq:r0}.
Define $\Phi_{H,h_0} \colon U\times \widetilde{S} \to A_{*}$ by
\begin{multline} \label{eq: Phi_H,h0}
\Phi_{H,h_0}(\zeta,p) := \left( \phi_{\zeta_{1,1}h_{1,1}(p)}^1 \circ \cdots \circ \phi_{\zeta_{N,1}h_{N,1}(p)}^N \circ \cdots \right. \\
\left. \cdots\circ \phi_{\zeta_{1,l}h_{1,l}(p)}^1 \circ \cdots \circ \phi_{\zeta_{N,l}h_{N,l}(p)}^N \right) \left( h_0(p) \right).
\end{multline}
Note that $\Phi_{H,h_0}(0,\cdot)=h_0$, the map $\Phi_{H,h_0}(\zeta,\cdot)$ is holomorphic on $\widetilde{S}$ for each $\zeta \in U$, $\Phi_{H,h_0} \approx \Phi_{H,\tilde{f}_0}$ uniformly on $U\times S$, and $\Phi_{H,h_0}$ agrees with $\Phi_{H,\tilde{f}_0}$ to order $r_0+1$ on $U\times ((E \cup \Lambda) \cap \mathring S)$.
Moreover, $\Phi_{H,h_0}(\zeta,\cdot)/g\in\mathscr{O}_\infty(\widetilde S|E,A_*)$  and $\Phi_{H,h_0}(\zeta,\cdot)/g-f \in \mathscr{A}_{D_2}(S)^n$. In particular, $\Phi_{H,h_0}(\zeta,\cdot)/g \in \mathscr{C}^0(C,f)$,
which enables us to consider the period map $\mathscr{R} \colon U \to \C^{nl}$ given by
\begin{equation*}
\mathscr{R}(\zeta) := \mathscr{P} \left(\Phi_{H,h_0}(\zeta,\cdot)/g \right).
\end{equation*}
If the approximation of $f_0$ by $h_0$ is sufficiently close on $S$, then $\mathscr{R}(\zeta) = \mathscr{P}(\Phi_{H,h_0}(\zeta,\cdot)/g) \approx \mathscr{P}(\Phi_{H,\tilde{f}_0}(\zeta,\cdot)/g) = \widetilde{\mathscr{Q}}(\zeta)$ for each $\zeta \in U$, implying that $\mathscr{R}$ is submersive at $0\in\C^{Nl}$ and $0\in\C^{nl}$ lies in the range of $\mathscr{R}$ by the implicit function theorem. Thus, there exists a parameter $\zeta_0 \in U$ close to $0$ such that $\mathscr{R}(\zeta_0)=0$.
Defining $\tilde{f} := \Phi_{H,h_0}(\zeta_0,\cdot)/g \in \mathscr{O}_{\infty}(\widetilde{S}|E,A_{*})$, it satisfies the following properties.
\begin{itemize}
\item
$\tilde{f}-f \in \mathscr{A}_{D_2}(S)^n$ and it is continuous on $S$. This implies condition~(iii) and the first part of~(i) in the statement of the lemma.

\item
$\mathscr{P}(\tilde{f})=\mathscr{R}(\zeta_0)=0$, and hence $(\tilde{f}-f)\theta$ is exact on $S$; see~\eqref{eq:period map} and recall that $C$ is a strong deformation retract of $S$. This shows~(ii).

\item
$\tilde{f} = \Phi_{H,h_0}(\zeta_0,\cdot)/g \approx \Phi_{H,\tilde{f}_0}(\zeta_0,\cdot)/g \approx f_0/g = f$ on $S$, proving~(i). 
\end{itemize}
This concludes Step $2$ and completes the proof of the lemma.
\end{proof}

\begin{proof}[Proof of Theorem~\ref{th:Mergelyan}]
The idea of the proof is just to pass to the derivative of the original directed meromorphic immersion, use Lemma~\ref{lemma:approx&interp into A, full}, and then integrate the map furnished by the lemma to obtain a directed meromorphic immersion meeting conclusions of the theorem.

Assume that the admissible set $S$ is connected; if not, apply the same argument on each of its connected components.
Moreover, assume that $K\neq \varnothing$; otherwise $S=\Gamma$ and $F\colon S\to \C^n$ is an $A$-immersion of class $\mathscr{C}^1(S)$ as in \cite[\textsection 7]{alarcon2014null}, hence \cite[Theorem 7.2-(a)]{alarcon2014null} solves the problem except for the interpolation, which is achieved as in what follows.
Fix a nowhere vanishing holomorphic $1$-form $\theta$ on $M$ and write $dF=f\theta$. Pick a point $p_0\in\mathring S\setminus (E\cup\Lambda)\subset \mathring K$ and let $\mathscr{C}=\{C_1,\ldots,C_l\}$ be a skeleton of $S$ based at $E\cup \Lambda\cup\{p_0\}$ (see Definition~\ref{def: skeleton}).
By Lemma~\ref{lemma:approx&interp into A, full}, we may approximate $f$ uniformly on $S$ by a full meromorphic map $\tilde{f}\in \mathscr{A}_\infty(\wt S|E,A_*)$ defined on a neighbourhood $\widetilde{S}$ of $S$ such that $S$ is a strong deformation retract of $\wt S$, $\tilde f$ agrees with $f$ to order $r$ on $E \cup (\Lambda\cap \mathring S)$, and $\mathscr{P}(\tilde f)=0$, where $\mathscr{P}$ is the period map~\eqref{eq:period map} associated to $\mathscr{C}$, $f$ and $\theta$. 
In particular, $\tilde f\theta$ is exact on $\wt S$; indeed, $f\theta=dF$ and $(\tilde f-f)\theta$ are exact on $C=\bigcup_{i=1}^lC_i$, and $C$ is a strong deformation retract of $S$, hence of $\wt S$ (see the last part in the proof of Lemma~\ref{lemma:approx&interp into A, full}).
Therefore, the formula
\begin{equation*}
G(p) := F(p_0) + \int_{p_0}^{p}\tilde{f}\theta, \quad p\in \widetilde{S}\setminus E,
\end{equation*}
defines a full meromorphic $A$-immersion on $\widetilde{S}$ arbitrarily close to $F$ on $S$, and such that the difference $G-F$ is continuous on $S$, vanishes at all points in $\Lambda$ and vanishes to order $r+1$ on $E \cup (\Lambda \cap \mathring S)$; take into account that $\mathscr{C}$ is based at $E\cup \Lambda\cup\{p_0\}$, $C$ is connected and $\mathscr{P}(\tilde f)=0$. We conclude that $G$ satisfies the theorem.
\end{proof}

\begin{remark} \label{remark:Mergelyan-extension}
If in Theorem~\ref{th:Mergelyan} we additionally assume that $A_{*}=A\setminus\{0\}$ is an Oka manifold, $S$ is connected and the inclusion $S \hookrightarrow M$ induces an isomorphism $H_1(S,\Z) \to H_1(M,\Z)$ of the homology groups, then there is a full meromorphic $A$-immersion $G \in\mathscr{I}_A(M|E)$ satisfying {\rm (i)} and {\rm (ii)} in the statement of Theorem~\ref{th:Mergelyan}.
Indeed, for this we adapt the proof of Lemma~\ref{lemma:approx&interp into A, full} as follows. 
Firstly, using Oka approximation \cite[Theorem~5.4.4]{forstneric2017stein} (see also \cite[Theorem~1.13.3]{alarcon2021minimal}), we approximate the map $\tilde{f}_0 \in \mathscr{A}(S,A_{*})$ in~\eqref{eq:tildef0} uniformly on $S$ by $h_0 \in \mathscr{O}(M,A_{*})$ such that $h_0-\tilde{f}_0$ vanishes to order $r_0+1$ on $E \cup (\Lambda \cap \mathring S)$.
Secondly, by the Mergelyan approximation theorem on admissible sets \cite[Theorem~16]{fornaess2020} (see also \cite[Theorem~1.12.11]{alarcon2021minimal}), we approximate each $h_{i,j} \in \mathscr{O}(S)$ in~\eqref{eq: h_ij} uniformly on $S$ by $\hat{h}_{i,j} \in \mathscr{O}(M)$ such that $\hat{h}_{i,j}-h_{i,j}$ vanishes to order $r_0+1$ on $E \cup (\Lambda \cap \mathring S)$. Write $\widehat{H} = (\hat{h}_{i,j})_{i,j}$.
The map $\Phi_{\widehat{H},h_0} \colon U \times M \to A_{*}$, defined as~\eqref{eq: Phi_H,h0} but using $h_0$ and $\hat{h}_{i,j}$ from above, satisfies $\Phi_{\widehat{H},h_0}(0,\cdot) = h_0$, $\Phi_{\widehat{H},h_0} \approx \Phi_{H,\tilde{f}_0}$ uniformly on $U\times S$, and their difference vanishes to order $r_0+1$ on $U \times (E \cup (\Lambda \cap \mathring S))$.
By the same argument as before, provided that all approximations are sufficiently close, the implicit function theorem furnishes us with a parameter $\zeta_0 \in U$ close to the origin, such that the full map $\tilde{f} := \Phi_{\widehat{H},h_0}(\zeta_0,\cdot)/g \in \mathscr{O}_{\infty}(M|E,A_{*})$ satisfies $\mathscr{P}(\tilde{f}) = \mathscr{P}(f)=0$, $\tilde{f} \approx f$ on $S$, and the difference $\tilde{f}-f$ vanishes to order $r$ on $E \cup (\Lambda \cap \mathring S)$.
The immersion $G$ is then obtained as in the proof of Theorem~\ref{th:Mergelyan}.
\end{remark}

\begin{proof}[Proof of Proposition~\ref{prop:semiglob}]
Let $F$ be as in the statement. Set $\varepsilon_0 := \varepsilon/2$ and fix a nowhere vanishing holomorphic $1$-form $\theta$ on $M$.
Remark~\ref{remark:Mergelyan-extension} provides a neighbourhood $U \subset M$ of $S$ with $H_1(S,\Z) \overset{\sim}{=} H_1(U,\Z)$, and a full meromorphic $A$-immersion $F^0 \in \mathscr{I}_A(U|E)$ such that $F^0-F$ is continuous on $S$, vanishes on $\Lambda \setminus \mathring S$, vanishes to order $r$ on $E \cup (\Lambda \cap \mathring S)$ and $||F^0-F||_{S}<\varepsilon_0$.

Choose a strongly subharmonic Morse exhaustion function $\tau \colon M \to \R$ with $S \subset \{\tau<0\} \subset \{\tau \leq 0\} \subset U$; see e.g.~\cite[Proposition~1.12.5]{alarcon2021minimal}. We may assume that $0$ is a regular value of $\tau$ and that every level set $\{\tau=c\}$, $c>0$, contains at most one critical point of $\tau$. 
Choose a strictly increasing sequence $0=c_0<c_1<c_2<\dots$ with $\lim_{j\to \infty}c_{j}= +\infty$ and such that for every $j\in \Z_+$, $c_j$ is a regular value of $\tau$, and the set $A_j = \{c_{j-1}<\tau<c_{j}\}$ ($j\in \N$) contains at most one critical point of $\tau$.
Set $M_j = \{\tau \leq c_j\}$, $j\in\Z_+$. We shall inductively construct a sequence of full meromorphic $A$-immersions $\{F^j \colon M_j\setminus E \to \C^n\}_j$, $F^j \in \mathscr{I}_A(M_j|E)$, and a decreasing sequence of positive numbers $\{\varepsilon_j\}_j$, satisfying the following conditions for all $j\in \N$.
\begin{itemize}
\item[\rm (1$_j$)] $F^{j}-F^{j-1}$ is continuous on $M_{j-1}$ and $||F^{j}-F^{j-1}||_{M_{j-1}}<\varepsilon_{j-1}$.

\item[\rm (2$_j$)] $F^j-F$ is continuous on $M_j$, vanishes on $\Lambda \setminus \mathring S$ and vanishes to order $r$ on $E \cup (\Lambda \cap \mathring S)$.

\item[\rm (3$_j$)] $0<\varepsilon_{j}<\varepsilon_{j-1}/2$, and if $G\colon M\setminus E \to \C^n$ is holomorphic with $||G-F^j||_{M_j\setminus E}<\varepsilon_{j-1}$, then $G|_{M_{j-1}\setminus E} \colon M_{j-1} \setminus E \to \C^n$ is an immersion.
\end{itemize}
If such sequences exist, then the sequence $\{F^j\}_j$ converges uniformly on compact sets to the limit map $G := \lim_{j\to \infty}F^{j} \colon M\setminus E \to \C^n$, which is a full meromorphic $A$-immersion of class $\mathscr{I}_A(M|E)$, the difference $G-F$ extends continuously to $S$, vanishes on $\Lambda \setminus \mathring S$ and vanishes to order $r$ on $E \cup (\Lambda \cap \mathring S)$, and $||G-F||_{S}<\varepsilon$. So, $G$ satisfies the conclusion of the theorem.

We now explain the induction.
The base of induction is provided by the already chosen $F^0 \colon U\setminus E \to \C^n$ and $\varepsilon_0$.
Assume that for some fixed integer $j\in \N$ we have already found maps $F^{i}$ and numbers $\varepsilon_i$ meeting the above conditions for all $i\in \{0,1,\dots,j-1\}$. To construct $F^{j}$, we consider the following two cases.

\textit{Case~1: The domain $A_j = \{c_{j-1}<\tau<c_j\}$ does not contain any critical point of $\tau$.}
In this situation, Remark~\ref{remark:Mergelyan-extension} gives a full meromorphic $A$-immersion $F^{j} \colon M_j\setminus E \to \C^n$ satisfying properties \rm (1$_j$) and \rm (2$_j$). Choosing $\varepsilon_j>0$ sufficiently small fulfils \rm (3$_j$) as well; note that here we use Cauchy estimates and the fact that $F^j$ has effective poles at all points in $E$, hence so does $G$.

\textit{Case~2: The domain $A_j = \{c_{j-1}<\tau<c_j\}$ contains a critical point $p$ of $\tau$.}
By assumption, $p$ is the unique critical point of $\tau$ on $A_j$ and it is a Morse point with Morse index either $0$ or $1$.
If the Morse index of $\tau$ at $p$ equals $1$, the change of topology is described by attaching to $M_{j-1}$ a smoothly embedded oriented arc $\Gamma_j \subset A_j \cup bM_{j-1}$ with an initial point $p_j \in bM_{j-1}$ and a final point $q_j \in bM_{j-1}$, and such that $\Gamma_j$ intersects $bM_{j-1}$ transversely and is otherwise contained in $A_j$. The set $S_j := M_{j-1} \cup \Gamma_j$ is admissible in $M_j$ and has the same topology as $M_j$.
If the points $p_j,q_j$ belong to the same connected component of $M_{j-1}$, a new nontrivial curve appears in the homology, however, if they belong to different components of $M_{j-1}$, we have that $H_1(M_{j-1},\Z) \overset{\sim}{=} H_1(S_j,\Z)$. In either case, we proceed as follows.
By~\cite[Lemma 3.5.4]{alarcon2021minimal}, we extend $f^{j-1}$ smoothly to the arc $\Gamma_j$ such that it still maps to $A_{*}$ and satisfies
\begin{equation*}
\int_{\Gamma_j} f^{j-1}\theta = F^{j-1}(q_j)-F^{j-1}(p_j).
\end{equation*}
The extended map $f^{j-1}$, defined on $S_j=M_{j-1} \cup \Gamma_j$, yields a meromorphic $A$-immersion $\widetilde{F}^{j-1} \in \mathscr{I}_A(S_j|E)$; note that $\widetilde{F}^{j-1} = F^{j-1}$ holds on $M_{j-1}$. Remark~\ref{remark:Mergelyan-extension} then furnishes the succeeding map $F^j \in \mathscr{I}_A(M_j|E)$.

Otherwise, the Morse index of $\tau$ at $p$ equals $0$. This time, a new connected component of $\{\tau \leq c \}$ appears at $p$ when $c$ passes $\tau(p)$. Let $\Delta \subset A_j$ denote a smoothly bounded disc around $p$. 
We extend $F^{j-1}$ to $S_j := M_{j-1}\cup \Delta$ by defining any full holomorphic $A$-immersion on $\Delta$. Since $S_j$ is admissible in $M_j$ and has the same topology as $M_j$, Remark~\ref{remark:Mergelyan-extension} applied to each component of $S_j$ furnishes a suitable map $F^{j} \in \mathscr{I}_A(M_j|E)$.
This closes the induction and completes the proof. 
\end{proof}


\section{Desingularizing meromorphic $A$-immersions} \label{sec:desingularizing}

In this section we prove the following more precise version of Corollary~\ref{co:intro-A1}-(b) for removing self-intersections of directed meromorphic immersions. 
We shall later use it in proofs of Theorems~\ref{th:Mittag-Leffler} and~\ref{th:Mittag-Leffler-properness} in order to ensure the embeddedness.
%
%
\begin{theorem}\label{th:compact-embedding}
If $M$ is a compact bordered Riemann surface, $E \subset \mathring M$ is a finite subset, and $A \subset \C^n$ $(n\ge 3)$ is a closed irreducible conical subvariety which is not contained in any hyperplane and is smooth away from $0$, then every meromorphic $A$-immersion $F \colon M\setminus E \to \C^n$ of class $\mathscr{I}_A(M|E)$ can be approximated uniformly on $M\setminus E$ by meromorphic $A$-embeddings $G \in \mathscr{I}_A(M|E)$ such that the difference map $G-F$ is continuous on $M$ and vanishes to any given finite order on $E$.

Moreover, if $\Lambda \subset M\setminus E$ is a finite set such that $F|_\Lambda \colon \Lambda \to \C^n$ is injective, then $G$ can be chosen to agree with $F$ on $\Lambda$ to any given finite order.
\end{theorem}
We emphasize that the approximation by embeddings is uniform in the non-compact surface $M\setminus E$.
We establish the theorem by a recursive application of the following result.
%
%
\begin{lemma} \label{lemma:emb}
Let $M$, $E$ and $A$ be as in Theorem~\ref{th:compact-embedding}.
Given an open neighbourhood $\Delta \subset \mathring M$ of $E$, every meromorphic $A$-immersion $F \colon M\setminus E \to \C^n$ of class $\mathscr{I}_A(M|E)$ can be approximated uniformly on $M\setminus E$ by a meromorphic $A$-immersion $G \in \mathscr{I}_A(M|E)$ such that $G|_{M\setminus \Delta}:M\setminus \Delta\to\C^n$ is injective, the difference map $G-F$ is continuous on $M$ and vanishes to any given finite order on $E$.

Moreover, if $\Lambda \subset M\setminus E$ is a finite set such that $F|_\Lambda \colon \Lambda \to \C^n$ is injective, then $G$ can be chosen to agree with $F$ on $\Lambda$ to any given finite order.
\end{lemma}
The key condition in the lemma is that, while the self-intersections are eliminated only from a compact subset, namely, $M\setminus \Delta$, the approximation is uniform on $M\setminus E$. The proof relies on the classical Abraham transversality argument~\cite{abraham1963transversality},~\cite[Theorem~1.4.3]{alarcon2021minimal} together with the method of sprays.
%
%
\begin{proof}[Proof of Theorem~\ref{th:compact-embedding} assuming Lemma~\ref{lemma:emb}]
We may assume that $M$ is a smoothly bounded compact domain in the interior of a compact bordered Riemann surface $\wt M$ and, by Theorem~\ref{th:Mergelyan}, that $F \in \mathscr{I}_A(\wt M|E)$.
Let $\Lambda\subset M\setminus E$ be a finite set such that $F|_\Lambda$ is injective.
Choose a sequence of open sets $\Delta_j \subset \mathring M$ ($j\in\N$) such that $E\subset \Delta_j$, $\Delta_j \cap \Lambda = \varnothing$ and $\overline \Delta_{j+1} \subset \Delta_j$ for all $j$, making sure that 
\begin{equation}\label{eq:Deltaj}
	\bigcup_{j\ge 1}\left(M \setminus \Delta_j\right) = M\setminus E.
\end{equation}
We choose each $\Delta_j$ to be a finite union of mutually disjoint smoothly bounded discs centred at the points in $E$.
Let $F^0:=F$ and fix $\varepsilon_0>0$.
We claim that there is a sequence of numbers $\varepsilon_j>0$ and meromorphic $A$-immersions $F^j \in \mathscr{I}_A(\wt M|E)$, $j\in\N$, such that the following hold for every $j\in \N$.
\begin{itemize}
\item[\rm (1$_j$)] $F^{j}-F^{j-1}$ is continuous on $\wt M$ and $||F^{j}-F^{j-1}||_{\wt M}<\varepsilon_{j-1}$.

\item[\rm (2$_j$)] $F^j-F$ is continuous on $\wt M$ and vanishes to any given finite order on $E\cup\Lambda$. 

\item[\rm (3$_j$)] $F^j$ is injective on $\wt M\setminus \Delta_j$.

\item[\rm (4$_j$)] $0<\varepsilon_j<\varepsilon_{j-1}/2$, and if $G\colon \wt M\setminus E \to \C^n$ is a holomorphic map with $||G-F^j||_{\wt M\setminus E}<2\varepsilon_j$, then $G$ is an embedding on $M\setminus \Delta_{j}$.
\end{itemize}
Indeed, we proceed by induction. The base is provided by the immersion $F=F^0$ and the number $\varepsilon_0$ chosen at the beginning of the proof. Observe that  {\rm (2$_0$)} holds.
For the inductive step, assume that for some fixed $j\in \N$ we have already found a number $\varepsilon_{j-1}>0$ and a map $F^{j-1}\in \mathscr{I}_A(\wt M|E)$ satisfying {\rm (2$_{j-1}$)}.
Lemma~\ref{lemma:emb} applied to $(\wt M, E, A, F^{j-1}, \Delta_{j}, \varepsilon_{j-1})$ then furnishes us with a meromorphic $A$-immersion $F^{j} \in \mathscr{I}_A(\wt M|E)$ satisfying \rm (1$_j$)--\rm (3$_j$). Then, in view of {\rm (3$_j$)}, the existence of a number $\varepsilon_j$ fulfilling {\rm (4$_j$)} is granted by Cauchy estimates. This closes the induction.

By {\rm (1$_j$)},~{\rm (2$_j$)},~{\rm (4$_j$)} and the properties of $\Delta_j$, there exists the limit map $G = \lim_{j\to\infty}F^j\colon \wt M\setminus E \to \C^n$, which is meromorphic on $\wt M$ with effective poles precisely in $E$ and satisfies that $G-F$ is continuous on $\wt M$.
In fact, for each $j\ge 0$ we have that $G-F^j$ is continuous on $\wt M$ and
\begin{equation} \label{eq:proof-emb-G-Fj}
||G-F^j||_{\wt M} \leq \sum_{k=j}^{\infty} ||F^{k+1}-F^{k}||_{\wt M} \stackrel{\text{\rm (1$_{k+1}$)}}{<} \sum_{k=j}^{\infty} \varepsilon_{k} \stackrel{\text{\rm (4$_{j}$)}}{<} 2\varepsilon_{j}.
\end{equation}
As a consequence, by property \rm (4$_j$), $G \in \mathscr{I}_A(M|E)$ and is injective on $M\setminus \Delta_{j}$ for all $j\ge 1$. By~\eqref{eq:Deltaj}, it turns out that $G$ is a meromorphic $A$-embedding on $M$.
Finally, condition~\rm (2$_j$) ensure that $G-F$ vanishes to the given finite order on $E\cup\Lambda$, whereas the approximation of $F$ by $G$ is guaranteed by~\eqref{eq:proof-emb-G-Fj}, closing the proof.
\end{proof}
%
%
%
\begin{proof}[Proof of Lemma~\ref{lemma:emb}]
The proof follows closely the arguments in that of~\cite[Theorem~2.4]{alarcon2014null} and~\cite[Theorem~3.4.1]{alarcon2021minimal} for the case $E=\varnothing$. We shall only point out the modifications which are necessary to deal with the poles and explain the first part of the lemma; the second part concerning interpolation on $\Lambda$ is established exactly as in the proof of~\cite[Theorem~3.4.1]{alarcon2021minimal} and we leave it out. 

So, let $\Delta \subset \mathring M$ be a small open neighbourhood of $E$ in $\mathring M$ as in the statement and assume that $\wt M := M\setminus \Delta$ is a smoothly bounded compact domain.
Define the difference map $\delta F \colon (M\setminus E) \times (M\setminus E) \to \C^n$ by
\begin{equation*} 
\delta F(p,q) := F(p)-F(q), \quad p,q\in M\setminus E.
\end{equation*}
Let $D_{\wt M} = \{(p,p) \colon p \in \wt M \}$ be the diagonal of $\wt M \times \wt M$. Since $F$ is an immersion on $M\setminus E$ and $\wt M\subset M\setminus E$ is compact, there exists an open neighbourhood $U \subset \wt M \times \wt M$ of $D_{\wt M}$ such that $\delta F$ does not take the value $0\in \C^n$ on $\overline{U}\setminus D_{\wt M}$.
Our aim is to find a meromorphic $A$-immersion $G \colon M\setminus E\to \C^n$ of class $\mathscr{I}_A(M|E)$ such that $G-F$ is continuous and arbitrarily close to $0$ uniformly on $M$, $G-F$ vanishes to the given finite order on $E$, and the difference map $\delta G$ is transverse to $0\in \C^n$ on $(\wt M \times \wt M)\setminus U$. By dimension reasons and the Abraham transversality argument~\cite{abraham1963transversality}, provided that $G$ is close enough to $F$ on $\overline U$, then $G$ results to be injective on $\wt M = M\setminus \Delta$.
For this, we need to find a neighbourhood $\Omega \subset \C^N$ of $0\in \C^N$ (for some large $N\in \N$) and a holomorphic map $H \colon (M\setminus E)\times \Omega \to \C^n$ meeting the following.
\begin{itemize}
\item
$H(\cdot,0)=F$ and for every $t\in\Omega$, the map $H(\cdot,t):M\setminus E\to\C^n$ is a meromorphic $A$-immersion of class $\mathscr{I}_A(M|E)$, the difference $H(\cdot,t)-F$ is continuous on $M$ and vanishes to the given finite order on $E$.

\item
The difference map $\delta H \colon (M\setminus E)\times (M\setminus E)\times \Omega \to \C^n$, defined by
	\begin{equation} \label{eq:diff-map} 
	\delta H(p,q,t) := H(p,t) - H(q,t), \quad (p,q,t) \in (M\setminus E) \times (M\setminus E) \times \Omega,
	\end{equation}
	is such that the differential map 
	\begin{equation*}
	\partial_{t}|_{t=0} \delta H\left(p,q,t\right) \colon \C^N \longrightarrow \C^n
	\end{equation*}
	is surjective for every $(p,q) \in (\wt M \times \wt M)\setminus U$.
\end{itemize}
That being the case, the map $G=H(\cdot,t)$ satisfies the conclusion for almost every parameter $t\in\Omega$ close enough to $0$.

The main step to achieve that is given by the following analogue of~\cite[Lemma~6.1]{alarcon2014null}.
%
%
\begin{claim}\label{cl:H(p,q)}
For any pair of points $p,q\in M\setminus E$, $p\neq q$, there is a deformation family $H=H^{(p,q)}(\cdot,t)$ as above, $t\in\C^n$, such that $\partial_t|_{t=0}\delta H(p,q,t):\C^n\to\C^n$ is an isomorphism.
\end{claim}
\begin{proof}
Let $\Upsilon\subset M\setminus E$ be a smooth embedded arc connecting $q$ to $p$. Choose a point $p_0\in M\setminus (E\cup\Upsilon)$ and a skeleton $\{C_1,\ldots,C_l\}$ of $M$ based at $E\cup\{p_0\}$ (see Definition~\ref{def: skeleton}) such that $C=\bigcup_{j=1}^l C_j\subset M\setminus \Upsilon$.  
Let $\theta$ be a holomorphic 1-form vanishing nowhere on $M$ and write $dF=f\theta$. 
Choose $g\in\Oscr(M)$ as in~\eqref{eq:(g)} and define $\tilde f_0$ as in~\eqref{eq:tildef0}, so $\tilde f_0\in\Oscr(M,A_*)$. A modification of the proof of~\cite[Lemma~6.1]{alarcon2014null}, using the ideas in the proof of Claim~\ref{claim:period spray}, provides a $\C^n$-valued $1$-form
\[
	\Theta_{\tilde f_0}(t,x,v),\quad x\in M,\; v\in T_xM,
\]
depending holomorphically on a parameter $t=(t_1,\ldots,t_n)\in\C^n$ on a neighbourhood of  $0\in\C^n$, such that for every such $t$ the difference 1-form $\Theta_{\tilde f_0}(t,\cdot,\cdot)/g-f\theta$ is holomorphic on $M$, it vanishes to the given finite order at every point in $E$, and
\[
	\int_{C_j}\left(\Theta_{\tilde f_0}(t,\cdot,\cdot)/g-f\theta\right)=0\quad \text{ for }j=1,\ldots,l
\]
(cf.~\eqref{eq:sprayQ} and~\eqref{eq:Qtilde}).
Moreover, $\Theta_{\tilde f_0}(t,\cdot,\cdot)/\theta$ takes its values in $A_*$, and, defining
\[
	H_F(x,t):=F(p_0)+\int_{p_0}^x\Theta_{\tilde f_0}(t,\cdot,\cdot)/g,\quad x\in M\setminus E
\]
(cf.~\cite[Eq.~(6.6)]{alarcon2014null}), it follows that $H_F(\cdot,0)=F$ and $H_F(\cdot,t):M\setminus E\to\C^n$ is a well-defined meromorphic $A$-immersion of class $\mathscr{I}_A(M|E)$ such that $H_F(\cdot,t)-F$ is continuous on $M$ and vanishes to the given finite order on $E$ for every $t\in\C^n$ sufficiently close to $0$. Furthermore, we may ensure that the vectors
\[
	\frac{\partial}{\partial t_i}\Big|_{t=0}\int_0^1\Theta_{\tilde f_0}\big(t,\upsilon(s),\dot\upsilon(s)\big)/g\in\C^n,\quad i=1,\ldots,n,
\]
are $\C$-linearly independent, where $\upsilon:[0,1]\to\Upsilon$ is a smooth parametrization of the arc $\Upsilon$ connecting $q$ to $p$. Since
\[
	\delta H_F(p,q,t)=H_F(p,t)-H_F(q,t)=\int_0^1\Theta_{\tilde f_0}\big(t,\upsilon(s),\dot\upsilon(s)\big)/g,
\]
the latter implies that the partial differential $\partial_t|_{t=0}\delta H_F(p,q,t)$ is an isomorphism. We refer to the proof of~\cite[Lemma~6.1]{alarcon2014null} for further details.
\end{proof} 

As $H^{(p,q)}$ satisfies that $\partial_t|_{t=0}\delta H^{(p,q)}(\tilde p,\tilde q,t):\C^n\to\C^n$ is an isomorphism for any pair of points $\tilde p\approx p$ and $\tilde q\approx q$ in $M\setminus E$, $\tilde p\neq \tilde q$, we may consider a finite open covering $\mathscr{U}=\{U_i\}_{i=1}^{m}$ of the compact set $(\wt M \times \wt M)\setminus U\subset M\times M$ and the corresponding families of maps $\{H^{i}(\cdot,t^{i}) \colon M\setminus E \to \C^n\}_{i=1}^m$, where $t^i \in \Omega_i \subset \C^n$  and $\Omega_i$ is a neighbourhood of $0 \in \C^n$, such that $H^i(\cdot,t^i)$ is a meromorphic $A$-immersion of class $\mathscr{I}_A(M|E)$, the difference $H^{i}(\cdot,t^{i})-F$ is continuous on $M$ and vanishes to any given finite order on $E$ for every $t^i \in \Omega_i$, $H^i(\cdot,0)=F$, and $\delta H^i(p,q,t^i) \colon \C^n \to \C^n$ is submersive at $t^i=0$ for every $(p,q) \in U_i$.
It turns out that the composition map
\begin{equation*}
H(p,t) = \left( H^1 \sharp \cdots \sharp H^m \right) (p,t^1,\dots,t^m), \quad p \in M\setminus E, \ t=(t^1,\dots,t^m) \in \C^N,
\end{equation*}
defined analogously to those in the proofs of \cite[Theorem~2.4]{alarcon2014null} and \cite[Theorem~3.4.1]{alarcon2021minimal}, satisfies $H(\cdot,0)=F$ and its difference map $\delta H$~\eqref{eq:diff-map} is submersive on $(\wt M \times \wt M)\setminus U$ for all $t=(t^1,\dots,t^m)$ close to $0\in \C^N$ (here, $N=mn$). Moreover, $H(\cdot,t)$ is a meromorphic $A$-immersion of class $\mathscr{I}_A(M|E)$, $H(\cdot,t)-F$ is continuous on $M$ and vanishes to any given finite order on $E$ for every $t\in \Omega \subset \C^N$ (where $\Omega \subset \C^N$ is a neighbourhood of $0$ which depends on all $\Omega_i$).

The proof is now finished by the aforementioned transversality argument.
\end{proof}


\section{The Mittag-Leffler theorem for meromorphic $A$-immersions} \label{sec:Mittag-Leffler}

In this section we establish the Mittag-Leffler theorem with approximation and interpolation for directed meromorphic immersions by proving the following more precise version of the first part of Theorem~\ref{th:intro-ML-DirMer}, including also condition~{\rm (iv)}. We shall complete the proof of Theorem~\ref{th:intro-ML-DirMer} by checking assertions~ {\rm (v)} and~{\rm (vi)}, concerning global geometry, in Section~\ref{sec:properness} (see Theorem~\ref{th:Mittag-Leffler-properness}).

\begin{theorem} \label{th:Mittag-Leffler}
Assume that $A\subset\C^n$ $(n\ge 3)$, $M$ and $E$ are as in Theorem~\ref{th:intro-ML-DirMer}. Let $U \subset M$ be a locally connected closed neighbourhood of $E$ whose connected components are all Runge admissible compact subsets in $M$, $\Lambda \subset U \setminus E$ be a closed discrete subset of $M$, $\varepsilon >0$ a number, and $r \colon E\cup \Lambda \to \N$ a map.
Given a map $F \colon U\setminus E \to \C^n$ such that $F|_{W\setminus (E\cap W)} \colon W\setminus (E\cap W) \to \C^n$ is of class $\mathscr{I}_A(W|E\cap W)$ for each component $W$ of $U$, there exists a full meromorphic $A$-immersion $G \colon M\setminus E \to \C^n$ of class $\mathscr{I}_A(M|E)$ satisfying the following conditions.
\begin{itemize}
\item[\rm (i)] $G-F$ is continuous on $E$, that is, it extends continuously to $U$.

\item[\rm (ii)] $G-F$ vanishes on $\Lambda \setminus \mathring U$ and vanishes to order $r(p)$ at $p \in E \cup (\Lambda \cap \mathring U)$.

\item[\rm (iii)] $||G-F||_U<\varepsilon$.
\end{itemize}
Furthermore, if $F$ is injective on $\Lambda$, then $G$ can be chosen an embedding.
\end{theorem}

\begin{proof}
Assume that $U$ has infinitely many connected components; otherwise the proof is simpler. By the assumptions, $U$ consists of countably many components which are Runge admissible compact sets, and there exists a normal exhaustion of $M$ by a sequence of connected smoothly bounded Runge compact domains
\begin{equation} \label{eq:exhaustion}
M_0 \Subset M_1 \Subset M_2 \Subset \dots \Subset \bigcup_{j=0}^{\infty}M_j = M,
\end{equation}
such that $M_0$ is a disc, $U\cap M_0 = \varnothing$, and for every $j \geq 1$, $U\cap bM_j = \varnothing$ and $M_j \setminus M_{j-1}$ contains exactly one component of $U$; call it $W_j$. In particular, $U\cap M_j=\bigcup_{i=1}^jW_i$ and $U=\bigcup_{i\ge 1}W_i$.

Fix a nowhere vanishing holomorphic $1$-form $\theta$ on $M$.
Choose an arbitrary full holomorphic $A$-embedding $F^0 \colon M_0 \to \C^n$ and write $dF^0 = f^{0}\theta$. (Note that $M_0 \cap U = \varnothing$.)
Set $\varepsilon_0 := \varepsilon/2$.
We shall inductively construct a decreasing sequence of positive numbers $\{\varepsilon_j\}_j$ and a sequence of full meromorphic $A$-immersions $\{F^j \colon M_j\setminus (E\cap M_j) \to \C^n\}_j$ of class $\mathscr{I}_A(M_j|E \cap M_j)$ such that, writing $dF^j = f^{j}\theta$, the following properties hold for all $j\geq 1$.
\begin{itemize}
\item[\rm (1$_j$)] $F^{j}-F^{j-1}$ is continuous on $M_{j-1}$ and $||F^{j}-F^{j-1}||_{M_{j-1}}<\varepsilon_{j-1}$.

\item[\rm (2$_j$)] $F^j-F$ is continuous on $U\cap M_j$ and $||F^{j}-F||_{U\cap M_{j}}<\varepsilon_{j-1}$.

\item[\rm (3$_j$)] $F^j-F$ vanishes on $(\Lambda \setminus \mathring U)\cap M_j$ and vanishes to order $r(p)$ at $p \in (E \cup (\Lambda \cap \mathring U))\cap M_j$.

\item[\rm (4$_j$)] If $F|_\Lambda \colon \Lambda \to \C^n$ is injective, then $F^j$ is injective on $M_j \setminus (E\cap M_{j})$.

\item[\rm (5$_j$)] $0<\varepsilon_j<\varepsilon_{j-1}/2$, and if $G\colon M\setminus E \to \C^n$ is a holomorphic map with $||G-F^j||_{M_j\setminus(E\cap M_{j})}<2\varepsilon_j$, then $G$ is an immersion on $M_{j-1} \setminus (E \cap M_{j-1})$.
\end{itemize}

The base of induction is provided by the already chosen $\varepsilon_0$ and $F^0$. (Note that $M_0\cap U = \varnothing$ and $F^0 \colon M_0 \to \C^n$ is a full holomorphic $A$-embedding.) Properties {\rm (1$_0$)} and {\rm (5$_0$)} are void, whereas \rm (2$_0$)--\rm (4$_0$) hold for free as $U\cap M_0 = \varnothing$.
For the inductive step, assume that for some fixed $j\in \N$ we have already found numbers $\varepsilon_i>0$ and full meromorphic $A$-immersions $F^i$ satisfying {\rm (2$_i$)}--{\rm (4$_i$)} for all $i \in \{0,1,\dots,j-1\}$. Let us construct $\varepsilon_j$ and $F^j$.
Firstly, we shall connect $M_{j-1}$ and $W_j\subset\mathring M_j \setminus M_{j-1}$ to obtain a compact connected admissible set $L_j$ which is Runge in $M$. 
For that, pick points $p_j \in bM_{j-1}$ and $q_j \in bW_j$ such that $q_j$ does not belong to any of the finitely many Jordan arcs in the description of $W_j$ as admissible set (see Definition~\ref{def: admissible set}). Take a smooth oriented arc $\Gamma_j \subset \mathring M_j$ with the endpoints $p_j$ and $q_j$ such that it intersects $M_{j-1}\cup W_j$ only at its endpoints $p_j, q_j$ and the intersection is transverse.
Apply \cite[Lemma~3.5.4]{alarcon2021minimal} and extend $f^{j-1} \colon M_{j-1}\setminus (E \cap M_{j-1}) \to A_{*}$ smoothly to $\Gamma_j$ such that it still maps to $A_{*}$ and satisfies
\begin{equation*} 
\int_{\Gamma_j} f^{j-1}\theta = F(q_j)-F^{j-1}(p_j).
\end{equation*}
Thus, we obtain an admissible set $L_j := M_{j-1} \cup \Gamma_j \cup W_j \subset \mathring M_j$ (see Figure~\ref{fig:induction ML}) and a map $\widetilde{F}^j$ on $L_j$, defined by
\begin{equation*}
\wt F^j(p) := \begin{cases}
			F^{j-1}(p) \  &\text{if }p\in M_{j-1}, \\
			F^{j-1}(p_j) + \int_{p_j}^{p} f^{j-1}\theta \ &\text{if }p\in \Gamma_j, \\
			F(p) \ &\text{if } p\in W_j.
		\end{cases}
\end{equation*}
By construction, $\widetilde{F}^j \in \mathscr{I}_A(L_j|E \cap L_j)$.
In this setting, Proposition~\ref{prop:semiglob} furnishes us with a full meromorphic $A$-immersion $F^j \colon M_j\setminus (E\cap M_j) \to \C^n$ of class $\mathscr{I}_A(M_j|E \cap M_j)$ such that $F^j-\widetilde{F}^j$ extends continuously to $L_j$, vanishes on $(\Lambda \setminus \mathring U)\cap L_j$ and vanishes to order $r(p)$ at $p \in (E \cup (\Lambda \cap \mathring U))\cap L_j$, and $||F^j-\widetilde{F}^j||_{L_j}<\varepsilon_{j-1}$. Furthermore, by Theorem~\ref{th:compact-embedding}, we may assume that $F^j$ is an $A$-embedding on $M_j \setminus (E\cap M_{j})$ provided that $F|_{\Lambda}$ is injective.
%
\begin{figure}[h!]
\begin{center}
\includegraphics[scale=0.48, trim={0 5cm 0 5cm}, clip]{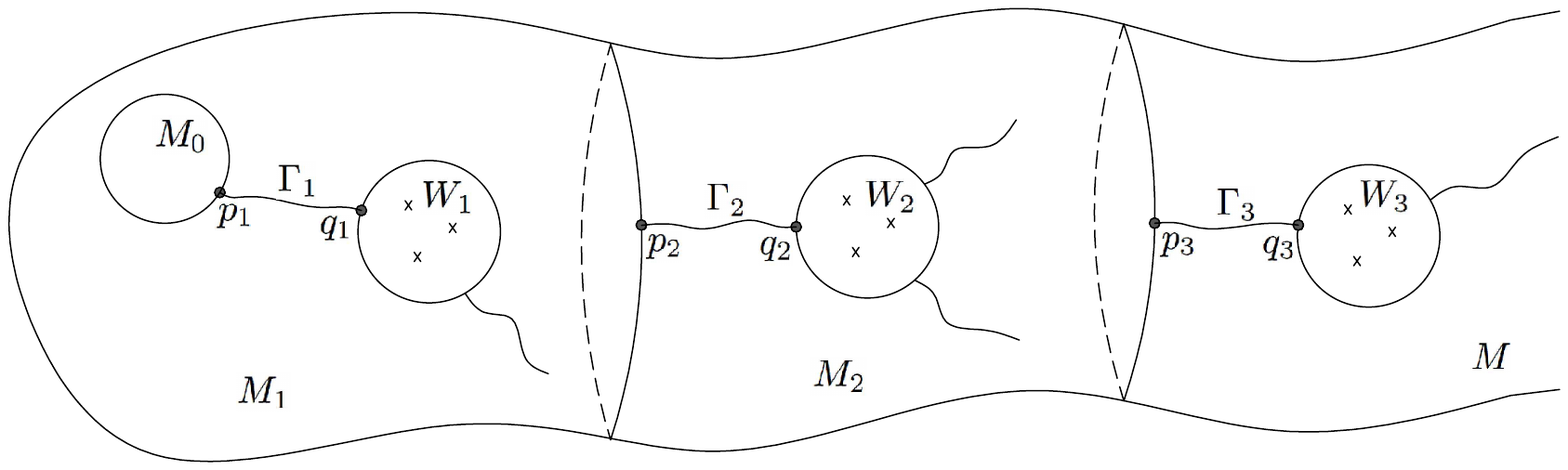}
\caption{Sets in the inductive process of the proof.}
\label{fig:induction ML}
\end{center}
\end{figure}

Let us check the properties.
Firstly, 
\begin{equation*}
||F^{j}-F^{j-1}||_{M_{j-1}} \leq ||F^j - \widetilde{F}^j||_{M_{j-1}} + ||\widetilde{F}^j-F^{j-1}||_{M_{j-1}} < \varepsilon_{j-1},
\end{equation*}
as $M_{j-1} \subset L_j$, which proves~\rm (1$_j$).
Similarly,
\begin{equation*}
||F^{j}-F||_{U \cap M_{j}} \leq ||F^j - \widetilde{F}^j||_{U \cap M_{j}} + ||\widetilde{F}^j-F||_{U \cap M_{j}} < \varepsilon_{j-1},
\end{equation*}
proving the second part of~\rm (2$_j$). (Note that in both inequalities above, the second summand equals $0$.) 
Furthermore,
$F^j-F = (F^j-\widetilde{F}^j)+(\widetilde{F}^j-F)$
extends continuously to $U\cap M_j$ since the first summand extends and the second one vanishes; by a similar argument, the difference vanishes on $(\Lambda \setminus \mathring U)\cap M_j$ and vanishes to order $r(p)$ at $p \in (E \cup (\Lambda \cap \mathring U))\cap M_j$, fulfilling~(2$_j$) and~(3$_j$).
Property~\rm (4$_j$) is already seen.
Finally, use Cauchy estimates to get a number $\varepsilon_j>0$ so small that~\rm (5$_j$) holds.
This closes the induction.

By~\eqref{eq:exhaustion},~(1$_j$) and~(3$_j$), there exists the limit map
\begin{equation*}
G := \lim_{j\to \infty} F^j \colon M\setminus E \longrightarrow \C^n
\end{equation*}
which is meromorphic on $M$ with effective poles precisely in $E$. 
We estimate
\begin{equation} \label{eq:proofML-G-Fj}
||G-F^j||_{M_j} \leq \sum_{k=j}^{\infty} ||F^{k+1}-F^{k}||_{M_j} \stackrel{\text{\rm (1$_{k+1}$)}}{<} \sum_{k=j}^{\infty} \varepsilon_{k} \stackrel{\text{\rm (5$_{j}$)}}{<} 2\varepsilon_{j},
\quad j\ge 0.
\end{equation}
By~\rm (5$_j$), $G \in \mathscr{I}_A(M_{j-1}|E \cap M_{j-1})$ for all $j\ge 1$, and hence $G \in \mathscr{I}_A(M|E)$. 
Moreover, the difference $G-F$ is continuous on $M$ by property~\rm (2$_j$), vanishes on $\Lambda \setminus \mathring U$ and vanishes to order $r(p)$ at $p \in E \cup (\Lambda \cap \mathring U)$ by~\rm (3$_j$).
To check the approximation, observe that
\begin{equation*}
||G-F||_{U\cap M_j} \leq ||G-F^j||_{U\cap M_j} + ||F^j-F||_{U\cap M_j} \stackrel{\text{\rm (2$_{j}$)}, \eqref{eq:proofML-G-Fj}}{<} 2\varepsilon_{j-1} \leq 2\varepsilon_0 = \varepsilon
\end{equation*}
holds for each $j \geq 1$, and, noting that $U\cap M_0 = \varnothing$, we conclude that $||G-F||_U < \varepsilon$, as desired. The fullness of $G$ is achieved for $\varepsilon_0>0$ small enough. Finally, in view of~{\rm (4$_j$)} and arguing as in the proof of Theorem~\ref{th:compact-embedding}, if $F|_\Lambda$ is injective, then we can grant that $G$ is an embedding by choosing each number $\varepsilon_j>0$ sufficiently small.
\end{proof}


\section{The Mittag-Leffler theorem for meromorphic $A$-immersions\\with fixed components} \label{sec:fixed components}

In this section we establish a version of Theorem~\ref{th:Mittag-Leffler} with the addition that, in the approximation, we keep fixed up to $n-2 \geq 1$ component functions of the given meromorphic directed immersion $M\to\C^n$, assuming they extend meromorphically to the whole open Riemann surface $M$. This generalizes \cite[Theorem~7.7]{alarcon2014null} to the meromorphic framework.

\begin{theorem} \label{th:Mittag-Leffler-fixed}
Let $M$, $E$, $U$, $A$, $r$ and $\varepsilon$ be as in Theorem~\ref{th:Mittag-Leffler}, $\Lambda \subset \mathring U\setminus E$ be a closed discrete subset of $M$, and $k\geq 1$ and $m\geq 2$ be integers. 
Set $n = k+m$, and assume that $A' = A \cap \{z_1=1,\dots,z_k=1\}$ is an Oka manifold and the projection $\pi = (\pi_1,\dots,\pi_k) \colon A \to \C^k$ onto the first $k$ coordinates
admits a local holomorphic section $\rho$ near $(z_1,\dots,z_k)=(0,\dots,0)$ with $\rho(0,\dots,0) \neq 0$.
Let $F=(F',F'') \colon U\setminus E \to \C^n$, where $F'=(F_1,\dots,F_k), \ F''=(F_{k+1},\dots,F_n)$, be as in Theorem~\ref{th:Mittag-Leffler}, and assume that $F'$ extends to a nonconstant meromorphic map $F' \colon M \to \C^k$ that is holomorphic on $M\setminus E$ and such that the zero set $(f_1,\dots,f_k)^{-1}(0,\dots,0)$ intersects $U$ only at points in $\mathring U$, where $f_i=dF_i/\theta$, $i\in \{1,\dots,n\}$, with $\theta$ a given holomorphic 1-form with no zeros on $M$.
Then there exists a meromorphic $A$-immersion $\widetilde{F}=(F',\widetilde{F}'') \in \mathscr{I}_A(M|E)$ satisfying the following.
\begin{itemize}
\item[\rm (i)] $\widetilde{F}''-F''$ extends continuously to $U$ and vanishes to order $r(p)$ at $p \in E \cup \Lambda$.

\item[\rm (ii)] $||\widetilde{F}''-F''||_U<\varepsilon$.
\end{itemize}
\end{theorem}

As in the general case, we obtain this theorem by a recursive application of the following approximation result of semiglobal type.
We omit the proof of Theorem~\ref{th:Mittag-Leffler-fixed} as it is analogous to the proof of Theorem~\ref{th:Mittag-Leffler}; the only difference being that we apply Proposition~\ref{prop:semiglob-fixed} in place of Proposition~\ref{prop:semiglob} in order to keep the first $k$-component functions fixed at all steps of the induction. 

\begin{proposition} \label{prop:semiglob-fixed}
Assume that $M$, $S$, $E$, $A$ and $r$ are as in Proposition~\ref{prop:semiglob}, $\Lambda \subset \mathring S \setminus E$ is a finite set, and $A'$, $\pi$, $k$, $m$ and $n$ are as is Theorem~\ref{th:Mittag-Leffler-fixed}.
Let $F = (F',F'') \colon S\setminus E \to \C^{n}$ be a map of class $\mathscr{I}_A(S|E)$ and assume that $F'$ extends to a nonconstant meromorphic map $F' \colon M \to \C^k$ that is holomorphic on $M\setminus E$ and such that the zero set $(f_1,\dots,f_k)^{-1}(0,\dots,0)$ intersects $S$ only at points in $\mathring S = \mathring K$, where $F'$, $F''$ and $f_i$ are defined as in Theorem~\ref{th:Mittag-Leffler-fixed}.
Given $\varepsilon >0$, there exists a meromorphic $A$-immersion $\widetilde{F}=(F',\widetilde{F}'') \in \mathscr{I}_A(M|E)$ satisfying the following.
\begin{itemize}
\item[\rm (i)] $\widetilde{F}''-F''$ extends continuously to $S$ and vanishes to order $r$ on $E \cup \Lambda$.

\item[\rm (ii)] $||\widetilde{F}''-F''||_S<\varepsilon$.
\end{itemize}
\end{proposition}

\begin{proof}
We combine ideas from the proof of Theorem~\ref{th:Mergelyan} with those of the proofs of \cite[Theorem~7.7]{alarcon2014null} and \cite[Theorem~3.7.1]{alarcon2021minimal}.
Firstly, observe that by assumptions on $A' = A \cap \{z_1=1,\dots,z_k=1\}$, the projection $\pi = (\pi_1,\dots,\pi_k) \colon A \to \C^k$  is a trivial fibre bundle with Oka fibre $A'$ except over $(0,\dots,0) \in \C^k$.
Given a nowhere vanishing holomorphic $1$-form $\theta$ on $M$, we write $dF = f\theta$ and $f=(f',f'')$ for $f'=(f_1,\dots,f_k)$ and $f''=(f_{k+1},\dots,f_n)$.
By assumption, $f'$ is meromorphic and nonconstant on $M$, hence its zero set $Q := (f')^{-1}(0) = \{q_1,q_2,\dots \}$ is discrete in $M$.
Take a holomorphic function $g\in \mathscr{O}(M)$ with 
\begin{equation}\label{eq:g-fixedC}
	(g) = \prod_{p\in E}p^{o(p)},
\end{equation}
where $o(p) = \max \{ \textrm{Ord}_{p}(f_i)\colon i=1,\dots,n\}$ and $\textrm{Ord}_{p}(f_i)$ denotes the order of $f_i$ at the pole $p\in E$.
Let 
\[
	h := fg
\]
and write $h=(h',h'')$ as above (cf.~\eqref{eq:(g)} and~\eqref{eq:tildef0}). It follows that $h \in \mathscr{A}(S,A_{*})$ and $(h')^{-1}(0) \subset Q \cup E$ is discrete in $M$. For simplicity of exposition, we may assume that $(h')^{-1}(0) = Q\cup E$.
Denote the pullback bundle $\Sigma := (h')^{*}A$. The pullback $(h')^{*}\pi \colon \Sigma \to M$ of the projection $\pi \colon A \to \C^k$ to $M$ is a trivial holomorphic fibre bundle over $M \setminus (h')^{-1}(0) = M\setminus (Q\cup E)$ with Oka fibre $A'$ and has singular fibre over points in $Q\cup E$.
Since $(h',h'')(x) \in \pi^{-1}(h'(x))$ for $x\in S$, the map $x \mapsto (x,h''(x))$ is a section of $\Sigma \to M$ over $S$. In the sequel, we shall say that $h''$ is a section of $\Sigma$.

To prove the proposition, we shall approximate $h''$ on $S$ by a global holomorphic section $\tilde{h}''$ of $\Sigma$ over $M$ meeting interpolation conditions on $E\cup \Lambda$. And we shall apply control on the periods to make sure that, defining 
\[
	\tilde{f}=(h',\tilde{h}'')/g:M\setminus E\longrightarrow A_*,
\] 
the meromorphic 1-form $\tilde f\theta$ on $M$ integrates to a meromorphic $A$-immersion $\widetilde{F} \colon M\setminus E \to \C^n$ satisfying the proposition.
In order to achieve that, we wish to apply \cite[Theorem~1.13.4]{alarcon2021minimal} (cf.~\cite[Theorem~6.14.6]{forstneric2017stein}); however, we need a globally defined continuous section of $\Sigma$ over $M$ which is holomorphic on a small neighbourhood of any point in $M$ over which $(h')^{*}\pi$ is ramified, i.e., on $Q\cup E$, and on a small neighbourhood of any point of interpolation, i.e., on $\Lambda \cup E$.
Note that if $p\in (Q\cup E\cup \Lambda)\cap S$, then $p\in \mathring S = \mathring K$ and $h''$ is holomorphic around $p$ by assumption.
The other possibility is that $p\in Q \cap (M\setminus S)$, because $E\cup \Lambda\subset \mathring S$. So, on a disc neighbourhood of any such point $p$ we choose a holomorphic section $h''$ such that $h''(p) \neq 0$ and $(h',h'')$ maps to $A_{*}$. We add all these disc neighbourhoods to the domain of holomorphicity of $h''$.

We aim at approximating $h''$ on $S$ and interpolating on $\widetilde{\Lambda}:=Q\cup E\cup \Lambda$ at the same time.
Observe that $(M, S\cup \widetilde{\Lambda})$ is homotopy equivalent to a relative CW-complex of dimension $1$, therefore, a skeleton of $M$ is obtained from $S\cup \widetilde{\Lambda}$ by attaching points and edges. Moreover, regular fibres $A'_{c_1,\dots,c_k} := A \cap \{z_1=c_1,\dots,z_k=c_k\}$ for $(c_1,\dots,c_k) \neq (0,\dots,0)$ are connected, which implies that any holomorphic section of $\Sigma$ over $S\cup \widetilde{\Lambda}$ can be extended to a continuous section over $M$. Hence, the assumptions in \cite[Theorem~1.13.4]{alarcon2021minimal} (the basic Oka principle for sections of ramified maps) are satisfied. 
We shall adapt the proof of Lemma~\ref{lemma:approx&interp into A, full} to the current aim.

Let $\mathscr{C} = \{C_1,\dots,C_l\}$ be a skeleton of $S$ based at $\widetilde{\Lambda}\cap S$ (see Definition~\ref{def: skeleton}). Recall that the union $C=\bigcup_{i=1}^{l}C_i$ is Runge in $S$ and contains a homology basis of $S$ and arcs enabling to interpolate on $\widetilde{\Lambda}\cap S$. 
Recall that $m=n-k\ge 2$, and let
\begin{equation}\label{eq:P-h'fixed}
	\mathscr{P} = \left(\mathscr{P}_1,\dots,\mathscr{P}_l \right) \colon \Big\{ \hat{h}''\in \mathscr{C}(C,\C^m) \colon \frac{\hat{h}''}{g}-f''\in \mathscr{C}^0(C,\C^m) \Big\} \longrightarrow \left(\C^m \right)^l,
\end{equation}
defined by
\begin{equation*}
\mathscr{P}_i (\hat{h}'') := \int_{C_i} \Big(\frac{\hat{h}''}{g}-f'' \Big) \theta \in \C^m, \quad \ i=1,\dots,l,
\end{equation*}
be the period map, applied to $\hat{h}''$; cf.~\eqref{eq:period map}. Throughout the proof, we shall consider holomorphic vector fields $V$ on $A$, vanishing at $0\in A$, which are tangential to the regular fibre $A'_{c_1,\dots,c_k} = A \cap \{z_1=c_1,\dots,z_k=c_k\}$ of the projection $\pi:A\to\C^k$ at every point $z\in A'_{c_1,\dots,c_k}$, $(c_1,\dots,c_k) \neq (0,\dots,0)$. 
The latter simply means that the vector field is tangential to $A$ and orthogonal to $\C^k\times\{(0,\ldots,0)\}\subset\C^n$. Note that if $\C\ni t\mapsto \phi(t,z)\in A$, $z\in A$, denotes the flow of such a $V$ for small values of $|t|$, then $\phi(0,z)=z$ for all $z\in A$ and 
\begin{equation}\label{eq:piphih'}
	\pi \left(\phi(t,h(p)) \right) = h'(p)\quad \text{for every $t$ and all $p\in S$}.
\end{equation}

We prove the proposition in five steps. Let $r'>r+\max\{o(p):p\in E\}$; see~\eqref{eq:g-fixedC}. 

\textit{Step~1: Approximating $h''$ by a full map.}
Following the arguments in the proof of Claim~\ref{cl:Step1}, but using vector fields of the above kind, the period map $\mathscr{P}$ in~\eqref{eq:P-h'fixed}, and the holomorphic section $h''$, we obtain a holomorphic map $\tilde{h}'' \in \mathscr{A}(S,\C^m)$, which is full on any curve $C_i \in \mathscr{C}$ and approximates $h''$ uniformly on $S$, such that $\tilde{h}''-h''$ vanishes to order $r'$ on $\widetilde{\Lambda} \cap S$, $\tilde h''/g-f''$ is continuous on $S$, $\mathscr{P}(\tilde{h}'') = 0$ and, taking into account~\eqref{eq:piphih'}, $(h',\tilde{h}'')(S) \subset A_{*}$. In particular, $\tilde h''$ is a section of $\Sigma$. Let us replace $h''$ by $\tilde{h}''$ and assume in what follows that $h''$ is full.

\textit{Step~2: Finding a period dominating spray.}
Now choose a family of holomorphic vector fields $V_1,\dots,V_N$ $(N\in \N)$ as above such that $\span\{V_1(z),\dots,V_N(z)\} = T_{z}(A'_{c_1,\dots,c_k})$ for all $z\in A'_{c_1,\dots,c_k}$ and $(c_1,\dots,c_k) \neq (0,\dots,0)$.
Set 
\[
	\Phi_{h'}(t,p):=h'(p)\in\C^k\quad \text{for all $t\in (\C^m)^l$ and $p\in S$.}
\] 
By~\eqref{eq:piphih'} and following the arguments in the proof of Claim~\ref{claim:period spray}, we can obtain a holomorphic spray $\Phi_{h''}:U\times S\to \C^m$ of sections of $\Sigma$ of class $\mathscr{A}(S,\C^m)$, where $U \subset (\C^m)^l$ is an open neighbourhood of $0 \in \C^{ml}$, such that the holomorphic spray
\begin{equation} \label{eq:spray-h',h''}
\Phi_{h} := \left(\Phi_{h'},\Phi_{h''}\right) \colon U\times S\longrightarrow \C^n
\end{equation}
assumes its values in $A_{*}$, $\Phi_{h}(0,\cdot)=h$, $\Phi_{h''}(t,\cdot)/g-f''$ is continuous on $S$ and vanishes to order $r'$ on $\widetilde{\Lambda} \cap S$ for all $t\in U$, and $\frac{\partial}{\partial t}|_{t=0} \mathscr{P}(\Phi_{h''}(t,\cdot)) \colon (\C^m)^l \to (\C^m)^l$ is an isomorphism; see~\eqref{eq:P-h'fixed}.

\textit{Step~3: The noncritical case.}
Assume that $S$ is a connected admissible set in an open Riemann surface $M$ such that the inclusion $S \hookrightarrow M$ induces an isomorphism $H_1(S,\Z) \to H_1(M,\Z)$. 
Let $\widetilde{\Lambda} \subset S$ be a finite subset and $\mathscr{C}$ be as above. (Note that at the moment, $S$ and $\widetilde{\Lambda}$ are not the ones in the statement of the proposition.)
Assume that $(h')^{*}\pi$ is as above, with its ramification values in $\mathring S = \mathring K$, and $h=(h',h'')$ and $\Phi_h$ are as in the previous steps for this pair $(M,S)$.
By \cite[Theorem~1.13.1, Theorem~1.13.4]{alarcon2021minimal}, there exists a holomorphic section $\hat{h}''$ of $\Sigma$ on $M$ which approximates $h''$ uniformly on $S$, and is such that $\hat h''/g-f''$ is continuous on $S$ and $\hat{h}''-h''$ vanishes to order $r'$ on $\widetilde{\Lambda}$.
Replace $h=(h',h'')$ by $\hat{h}=(h',\hat{h}'') \colon M \to A_{*}$ as the core of the spray $\Phi_h$ in~\eqref{eq:spray-h',h''}, then interpolate and approximate uniformly on $S$ all functions used in the construction of $\Phi_h$ by global holomorphic functions on $M$ (see the proof of Claim~\ref{claim:period spray}).
The new spray $\Phi_{\hat{h}} \colon \wh U \times M \to A_{*}$ (where $0 \in \wh U \subset U \subset \C^{ml}$), given by $\Phi_{\hat{h}} = (\Phi_{h'}, \Phi_{\hat{h}''})$, has the same properties as $\Phi_h$ provided that all approximations are close enough.
Consequently, the implicit function theorem guarantees existence of a point $t_0 \in \wh U$ near the origin such that the map $\tilde{h}'' := \Phi_{\hat{h}''}(t_0,\cdot)$ is a section of $M \to \Sigma$ satisfying $\mathscr{P}(\tilde{h}'')=0, \ \tilde{h}''\approx h''$ on $S$, $\tilde h''/g-f''$ is continuous on $S$ and $\tilde{h}''-h''$ vanishes to order $r'$ on $\widetilde{\Lambda}$.

\textit{Step~4: The induction and the critical case.}
We proceed by induction as in the proof of Proposition~\ref{prop:semiglob}. We shall briefly explain the procedure at the $j$-th step ($j \geq 1$) where we have a suitable meromorphic section $(h'')^{j-1}$ of $\Sigma$ defined on the compact set $M_{j-1}$ and wish to approximate it with interpolation by a section $(h'')^j$ defined on $M_j$ (here, the sets $M_j$ form an exhaustion of $M$ as in the proof of Proposition~\ref{prop:semiglob}). We distinguish two cases. In the first one, the annulus $A_j$ (defined as in the proof of Proposition~\ref{prop:semiglob}) does not contain any critical points of the Morse exhaustion function $\tau$ on $M$; this case is solved by the noncritical case. On the contrary, when $p\in A_j$ is the unique critical point of $\tau$ on $A_j$, as described in the proof of the cited proposition, we add certain arcs (denoted by $\Gamma_j$) if necessary. We keep the first $k$ component functions fixed (i.e., $h'=(h')^{j}$), and apply the argument to $(h'')^{j-1}$, such that the extension of $(h'')^{j-1}$ to $\Gamma_j$ is a holomorphic section of $\Sigma$ and $\int_{\Gamma_j} \frac{(h'')^{j-1}}{g} \theta$ equals a prescribed vector in $\C^m$.
(Note that the use of an analogue of Gromov integration lemma is justified by \cite[Remark~3.5.5]{alarcon2021minimal}).

\textit{Step~5. Completion of the proof.}
Performing the induction process, we then obtain in the limit a holomorphic map $\tilde h''=\lim_{j\to\infty} (h'')^j:M\to\C^m$, such that $(h',\tilde h''):M\to\C^n$ takes its values in $A_*$, $\tilde h''/g$ is holomorphic on $M\setminus E$, the difference $\tilde h''-h''$ is uniformly close to $0$ on $S$ and it vanishes to order $r'$ at every point in $E\cup\Lambda$. Thus, $\tilde h''/g-f''$ is continuous on $S$ and vanishes to order $r$ everywhere on $E\cup\Lambda$. Furthermore, by the control of periods, we have that $(\tilde h''/g)\theta$ is exact on $M\setminus E$ and the holomorphic map $\widetilde F'':M\setminus E\to \C^m$, defined by 
\[
	\widetilde F''(p):=F''(p_0)+\int_{p_0}^p(\tilde h''/g)\theta,\quad p\in M\setminus E,
\]
for any given base point $p_0\in S\setminus E$,
satisfies that $\widetilde F''-F''$ is continuous and $\varepsilon$-close to $0$ on $S$, and it vanishes to order $r$ at every point in $E\cup\Lambda$. It is now clear that the well-defined meromorphic $A$-immersion $\widetilde F:=(F',\widetilde F''):M\setminus E\to\C^n$ of class $\mathscr{I}_A(M|E)$ satisfies the conclusion of the proposition.
\end{proof}


\section{Global properties and completion of the proof of Theorem~\ref{th:intro-ML-DirMer}} \label{sec:properness}

Recall that an immersion $f\colon M \to \C^n$ from an open Riemann surface $M$ is \emph{complete} if the path $f\circ \gamma$ has infinite Euclidean length in $\C^n$ for any divergent path $\gamma$ in $M$. (A path $\gamma \colon [0,1) \to M$ is \emph{divergent} if $\gamma(t) \notin K$ for any given compact set $K \subset M$ as $t\to 1$.)
A continuous map $f\colon X \to Y$ of topological spaces is \emph{almost proper} if for every compact set $K \subset Y$ the connected components of $f^{-1}(K)$ are all compact.
If for every such $K \subset Y$ the preimage $f^{-1}(K)$ is compact in $X$, then the map $f$ is \emph{proper}.
Every proper map is also almost proper, and every almost proper immersion $M\to\C^n$ is complete, whereas the contrary implications do not hold.

In this section, we shall complete the proof of Theorem~\ref{th:intro-ML-DirMer} by checking the global properties in assertions {\rm (v)} and {\rm (vi)}. This is granted by the following more precise statement, which is an extension of the Mittag-Leffler theorem for $A$-immersions in Theorem~\ref{th:Mittag-Leffler}.
%
\begin{theorem} \label{th:Mittag-Leffler-properness}
Let $M$, $E$, $U$, $\Lambda$, $A$, $r$, $n$, $\varepsilon$ and $F$ be as in Theorem~\ref{th:Mittag-Leffler}.
Then there exists a meromorphic $A$-immersion $G \in \mathscr{I}_A(M|E)$ satisfying the conclusion of Theorem~\ref{th:Mittag-Leffler} and also the following properties.
\begin{itemize}
\item[\rm (i)] If in addition $A \cap \{z_1=1\}$ is an Oka manifold and the coordinate projection $\pi_1 \colon A \to \C$ onto the $z_1$-axis admits a local holomorphic section $\rho_1$ near $z_1=0$ with $\rho_1(0) \neq 0$, then $G$ can be chosen an almost proper map, and hence a complete immersion.

\item[\rm (ii)] If in addition $A \cap \{z_i=1\}$ is an Oka manifold for every $i \in \{1,\dots,n\}$, each coordinate projection $\pi_i \colon A \to \C$ onto the $z_i$-axis admits a local holomorphic section $\rho_i$ near $z_i=0$ with $\rho_i(0) \neq 0$, and either $U$ has finitely many connected components or $U=\bigcup_{j\in\N}W_j$ and
\begin{equation} \label{eq:mu_j}
\lim_{j\to \infty} \min \left\{|F(p)| \colon p \in W_j \right\} = +\infty,
\end{equation} 
where $\{W_j\}_{j\in\N}$ is the family of connected components of $U$,
then $G$ can be chosen a proper map.
\end{itemize}
\end{theorem}
The proof relies on improving the inductive construction in the proof of Theorem~\ref{th:Mittag-Leffler} by implementing the following two results, which are the new key ingredients. The first lemma concerns almost proper maps, and the second lemma the proper ones.

\begin{lemma} \label{lemma:almost-properness}
Let $M$, $A$ and $n$ be as in Theorem~\ref{th:Mittag-Leffler-properness}-(i), assume that $K$ and $L$ are smoothly bounded Runge compact domains in $M$ such that $K \subset \mathring L$, and $E, \Lambda \subset \mathring K$ are finite disjoint sets.
Let $F \colon K\setminus E \to \C^n$ be a meromorphic $A$-immersion of class $\mathscr{I}_A(K|E)$. Then for any numbers $r \in \N$, $\varepsilon>0$ and $\delta >0$, there is a meromorphic $A$-immersion $\widetilde{F}=(\widetilde{F}_1,\widetilde{F}_2,\ldots,\widetilde{F}_n) \in \mathscr{I}_A(L|E)$ such that the following hold.
\begin{itemize}
\item[\rm (i)] $\widetilde{F}-F$ is continuous on $K$.

\item[\rm (ii)] $\widetilde{F}-F$ vanishes to order $r$ on $E\cup \Lambda$.

\item[\rm (iii)] $||\widetilde{F}-F||_{K} < \varepsilon$.

\item[\rm (iv)] $\max \{|\wt{F}_1(p)|, |\wt{F}_2(p)|\} > \delta$ for all $p \in bL$.
\end{itemize}
\end{lemma}

\begin{proof}
For simplicity of exposition, assume that $bL$ is connected, hence a Jordan curve. 
Pick an open neighbourhood $\wt L$ of $L$ in $M$ and choose closed discs $\Delta, \Omega \subset \wt L \setminus K$ such that $bL \subset \Delta \cup \Omega$. Note that $\Delta \cap \Omega \neq \varnothing$ and assume that it is the union of two disjoint discs; see Figure~\ref{fig:lemma-almost-prop}.
%
\begin{figure}[h!]
\begin{center}
\includegraphics[scale=0.4, trim={0 7cm 0 9.75cm},clip]{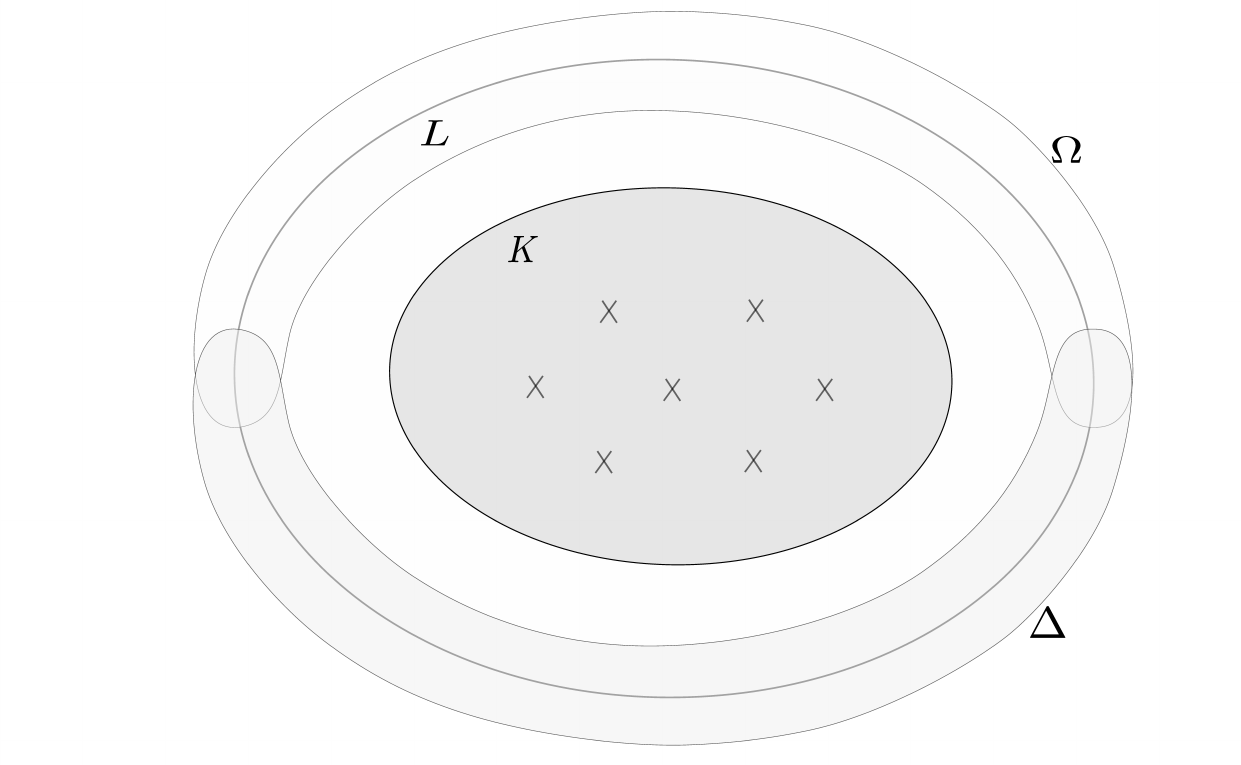}
\caption{Discs $\Delta$ and $\Omega$ in the proof of Lemma~\ref{lemma:almost-properness}.}
\label{fig:lemma-almost-prop}
\end{center}
\end{figure}

Firstly, consider a map $G=(G_1,\dots,G_n)\colon (K\setminus E) \cup \Delta \to \C^n$ of class  $\mathscr{I}_A(K\cup \Delta|E)$ such that $G=F$ on $K\setminus E$ and $|G_1(p)|>\delta+\varepsilon$ for all $p\in \Delta$. By Proposition~\ref{prop:semiglob}, there exists a map $\wt G = (\wt G_1,\dots,\wt G_n) \colon \wt L \setminus E \to \C^n$ of class $\mathscr{I}_A(\wt L|E)$ such that the difference map $\wt G-G$ extends continuously to $K\cup \Delta$, vanishes to order $r$ on $E\cup \Lambda$, and satisfies $|\wt G-G|<\varepsilon/2$ on $K\cup \Delta$. Therefore, $|\wt G_1(p)|>\delta$ for every $p\in \Delta$.
Secondly, define a map $H=(H_1,\dots,H_n)\colon (K\setminus E) \cup \Omega \to \C^n$ of class  $\mathscr{I}_A(K\cup \Omega|E)$ by $H = \wt G$ on $K\setminus E$  and $H = (\wt G_1, \wt G_2+c, \wt G_3,\dots,\wt G_n)$ on $\Omega$, where $c>0$ is so large that $|H_2(p)|>\delta+\varepsilon$ for all $p\in \Omega$.
Proposition~\ref{prop:semiglob-fixed} furnishes a map $\wt F = (\wt F_1,\dots,\wt F_n) \colon \wt L \setminus E \to \C^n$ of class $\mathscr{I}_A(\wt L|E)$ such that $\wt F-H$ extends continuously to $K\cup \Omega$, vanishes to order $r$ on $E\cup \Lambda$, satisfies $|\wt F-H|<\varepsilon/2$ on $K\cup \Omega$ and $\wt F_1=\wt G_1$ on $\wt L$.

It is clear that $\wt F$ satisfies the conclusion of the lemma.
In particular, we have that  $\widetilde{F}-F$ vanishes to order $r$ on $E\cup \Lambda$, $|\wt F-F| \leq |\wt F-H|+|\wt G-F| < \varepsilon$ on $K$, $|\wt F_1(p)|=|\wt G_1(p)|>\delta$ for all $p\in \Delta$, and $|\wt F_2(p)|>|H_2(p)|-\varepsilon/2>\delta$ for all $p\in \Omega$.
\end{proof}

\begin{lemma} \label{lemma:properness} 
Let $M$, $A$ and $n$ be as in Theorem~\ref{th:Mittag-Leffler-properness}-(ii), and assume that $K$, $L$, $E$ and $\Lambda$ are as in Lemma~\ref{lemma:almost-properness}.
Let $F \colon K\setminus E \to \C^n$ be a meromorphic $A$-immersion of class $\mathscr{I}_A(K|E)$ mapping $bK$ into $\C^n \setminus \{0\}$. Let $r \in \N$ and $\varepsilon>0$, and pick a number
\begin{equation*} \label{eq:lemproper-delta}
0 < \delta < \min \left\{ ||F(p)||_{\infty} \colon p \in bK \right\},
\end{equation*}
where $||\cdot||_\infty$ denotes the infinity norm in $\C^n$.
Then there exists a meromorphic $A$-immersion $\widetilde{F} \in \mathscr{I}_A(L|E)$ such that the following hold.
\begin{itemize}
\item[\rm (i)] $\widetilde{F}-F$ is continuous on $K$.

\item[\rm (ii)] $\widetilde{F}-F$ vanishes to order $r$ on $E\cup \Lambda$.

\item[\rm (iii)] $||\widetilde{F}-F||_{K} < \varepsilon$.

\item[\rm (iv)] $||\widetilde{F}(p)||_{\infty} > \delta$ for all $p \in L \setminus \mathring K$.

\item[\rm (v)] $||\widetilde{F}(p)||_{\infty} > 1/\varepsilon$ for all $p \in bL$.
\end{itemize}
\end{lemma}
The proof is very similar to that of \cite[Lemma~3.11.1]{alarcon2021minimal} (see also \cite[Lemma~8.2]{alarcon2014null}) and we leave the details out. The main difference is that we use Propositions~\ref{prop:semiglob} and~\ref{prop:semiglob-fixed} in place of \cite[Theorems~3.6.1 and~3.7.1]{alarcon2021minimal}. Moreover, in the inductive process, we ask the map $F^b$, constructed at the $b$-th step of induction, to satisfy analogous properties to \rm(C1$_b$)--\rm(C7$_b$) except \rm(C3$_b$) in the cited proof, and also that the difference $F^{b}-F^{b-1}$ is continuous on $L$ and vanishes to order $r$ on $E \cup \Lambda$, which is possible by our mentioned results.

\begin{proof}[Proof of Theorem~\ref{th:Mittag-Leffler-properness}]
We shall begin by carefully explaining the proof of assertion~(ii). The proof of part~(i), being less involved, follows the same line of arguments and is briefly discussed at the end.
Without loss of generality, we may assume that $U$ is an infinite union of connected components $\{W_j\}_{j\in\N}$; the proof is again simpler otherwise.
Set
\begin{equation}\label{eq:muj7}
\mu_j := \min \left\{||F(p)||_{\infty} \colon p \in W_j \right\}-1\ge -1
\end{equation} 
and note that, by~\eqref{eq:mu_j}, $\lim_{j\to \infty}\mu_j = +\infty$.
As in the proof of Theorem~\ref{th:Mittag-Leffler}, take a normal exhaustion $\{M_j\}_j$ of $M$ satisfying the same properties as~\eqref{eq:exhaustion}, and assume that $W_j$ is the unique component of $U$ lying in $M_{j}\setminus M_{j-1}$ for every $j\geq 1$.
Choose a full holomorphic $A$-embedding $F^0 \colon M_0 \to \C^n$ such that $F^0$ has no zeros on $bM_0$, let $0<\delta_0<\min\{||F^0(p)||_{\infty} \colon p\in bM_0\}$ and $\varepsilon_0 := \varepsilon/2$.
We shall inductively construct a decreasing sequence of positive numbers $\{\varepsilon_j\}_j$, a sequence of positive numbers $\{\delta_j\}_j$ with $\lim_{j\to \infty}\delta_j=+\infty$, and a sequence of full meromorphic $A$-immersions $\{F^j\colon M_j\setminus (E\cap M_j) \to \C^n\}_j$ of class $\mathscr{I}_A(M_j|E \cap M_j)$, such that properties (1$_j$)--(5$_j$) in the proof of Theorem~\ref{th:Mittag-Leffler} as well as the following ones hold for all $j\ge 1$.
\begin{itemize}
\item[\rm (6$_j$)] $\delta_j > \min\{j,\delta_{j-1},\mu_j\}$.

\item[\rm (7$_j$)] $||F^j(p)||_{\infty} > \delta_j$ for all $p\in bM_j$.

\item[\rm (8$_j$)] $||F^j(p)||_{\infty} > \min\{\delta_{j-1},\mu_j\}$ for all $p\in M_{j}\setminus \mathring M_{j-1}$.
\end{itemize}
The base of induction is provided by the triple $(F^0, \delta_0, \varepsilon_0)$. (Observe that (6$_0$) and (8$_0$) are void.)
For the inductive step, assume that for some fixed integer $j\geq 1$, we have already found $(F^{i}, \delta_i, \varepsilon_i)$ for all $i \in \{0,\dots,j-1\}$, satisfying properties \rm(2$_i$)--\rm(4$_i$) and \rm(7$_i$).
Following the ideas in the proof of Theorem~\ref{th:Mittag-Leffler}, we construct the admissible set $L_j = M_{j-1}\cup \Gamma_j \cup W_j$ (see Figure~\ref{fig:induction ML}) and extend $F^{j-1}$ to $\Gamma_j$ with the same properties as in the cited proof, while ensuring in addition that 
\begin{equation}\label{eq:muj72}
||F^{j-1}||_{\infty}>\min \left\{\delta_{j-1},\mu_j \right\} \quad \textrm{on } \Gamma_j,
\end{equation} 
as we may by~\eqref{eq:muj7} and~(7$_{j-1}$).
Call this extended map $\widetilde{F}^j$; note that $\wt F^j \in \mathscr{I}_A(L_j|E \cap L_j)$, $E\cap L_j = E \cap M_j$, $\widetilde{F}^j = F$ on $W_j$, and $\widetilde{F}^j = F^{j-1}$ on $M_{j-1}\cup \Gamma_j$.
By Proposition~\ref{prop:semiglob}, assume that $\widetilde{F}^j$ is a meromorphic $A$-immersion on a small smoothly bounded compact neighbourhood $\widetilde{M}_{j-1} \subset M_j$ of $L_j$ (see Figure~\ref{fig:induction ML-prop}), satisfying the required interpolation and approximation conditions, and such that
\begin{equation} \label{proof:ML-P-tildeFj}
||\widetilde{F}^j||_{\infty} > \min\{\delta_{j-1},\mu_j\} \quad \textrm{on } \widetilde{M}_{j-1}\setminus \mathring M_{j-1};
\end{equation}
see~\eqref{eq:muj7},~(7$_{j-1}$) and~\eqref{eq:muj72}.
Pick $\delta_j$ satisfying (6$_j$). In view of~\eqref{proof:ML-P-tildeFj}, Lemma~\ref{lemma:properness} and Theorem~\ref{th:compact-embedding} furnish us with a meromorphic $A$-immersion $F^j \colon M_j\setminus (E\cap M_j) \to \C^n$ satisfying (1$_j$)--(4$_j$) and (6$_j$)--(8$_j$), whilst condition (5$_j$) is fulfilled by choosing $\varepsilon_j>0$ small enough. This closes the induction. 
%
\begin{figure}[h!]
\begin{center}
\includegraphics[scale=0.45, trim={0 4cm 0 5cm},clip]{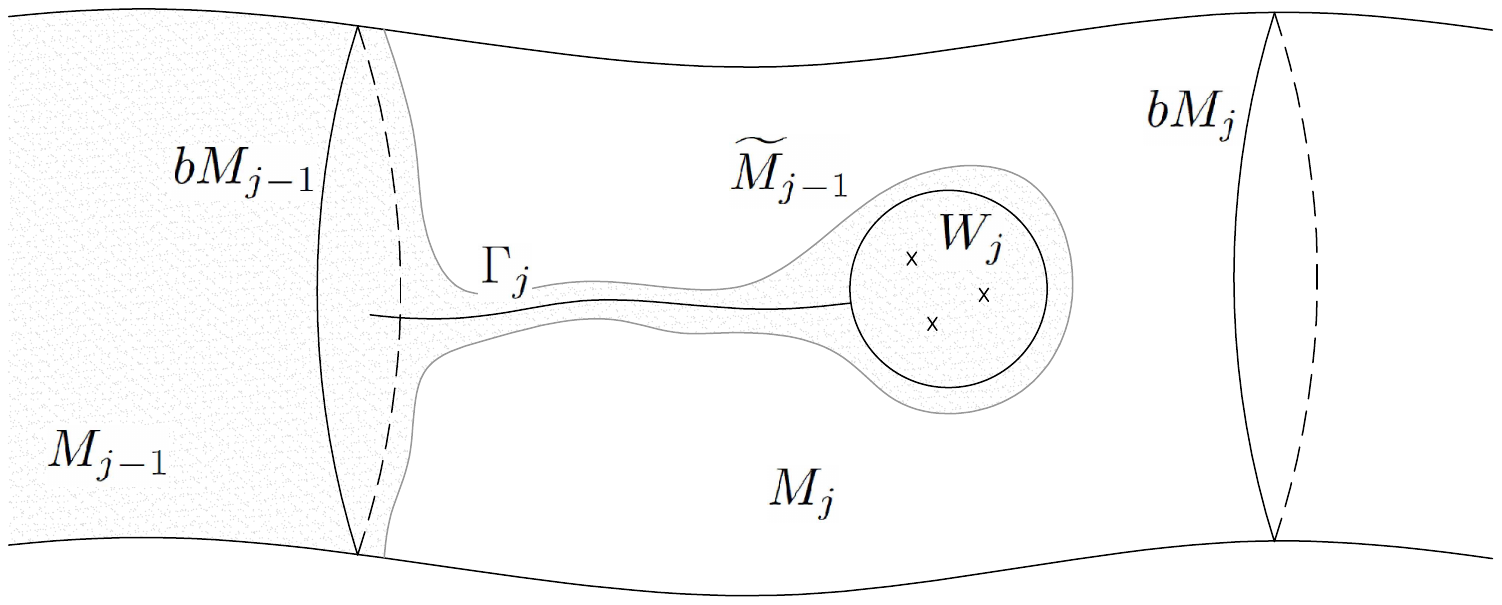}
\caption{Sets used at the $j$-th step of induction.}
\label{fig:induction ML-prop}
\end{center}
\end{figure}

By the proof of Theorem~\ref{th:Mittag-Leffler}, the limit map
\begin{equation*}
G := \lim_{j\to \infty} F^j \colon M\setminus E \longrightarrow \C^n
\end{equation*}
exists and satisfies the conclusion of that theorem. In particular, we have that 
\begin{equation} \label{eq:proof:ML-P-G-Fj}
	||G-F^j||_{M_j}<2\varepsilon_{j} \leq \varepsilon, \quad j\geq 0;
\end{equation}
see~\eqref{eq:proofML-G-Fj}. It remains to check that $G:M\setminus E\to\C^n$ is a proper map.
So, take a divergent sequence $\{p_m\}_m \subset M\setminus E$. If $\{p_m\}_m$ diverges on $M$, then~\eqref{eq:exhaustion},~\eqref{eq:proof:ML-P-G-Fj} and (8$_j$) imply that $\{G(p_m)\}_m$ diverges in $\C^n$; recall that $\lim_{j\to\infty}\delta_j=\lim_{j\to\infty}\mu_j=+\infty$. If on the contrary, $\{p_m\}_m$ admits a subsequence $\{q_m\}_m$ converging to a point in $E$, then $\{G(q_m)\}_m$ diverges in $\C^n$ as well since $G$ has an effective pole at every point in $E$. This shows that $G$ is proper, concluding the proof of~(ii).

For assertion~(i), apply Lemma~\ref{lemma:almost-properness} instead of Lemma~\ref{lemma:properness} and omit condition (8$_j$) in the induction process. By (7$_j$) and~\eqref{eq:proof:ML-P-G-Fj}, it follows that $||G(p)||_\infty>\delta_j-\varepsilon$ for all $p\in bM_j$, which goes to $+\infty$ as $j\to +\infty$. Taking into account~\eqref{eq:exhaustion} and that $G$ has an effective pole at every point in $E$, this implies that $G:M\setminus E\to\C^n$ is an almost proper map.
\end{proof}

A simplification of the proof of Theorem~\ref{th:Mittag-Leffler-properness} also gives the following.
%
%
\begin{proposition}\label{co:ML-A-proper}
Let $A$, $M$ and $E$ be as in Theorem~\ref{th:Mittag-Leffler-properness}, and assume that for $i=1,2$, $A\cap\{z_i=1\}$ is an Oka manifold and the coordinate projection $\pi_i \colon A \to \C$ onto the $z_i$-axis admits a local holomorphic section $\rho_i$ near $z_i=0$ with $\rho_i(0) \neq 0$. Then there exists a meromorphic $A$-embedding $F=(F_1,\ldots,F_n):M\setminus E\to\C^n$ of class $ \mathscr{I}_A(M|E)$ such that $(F_1,F_2):M\setminus E\to\C^2$ is proper.
\end{proposition}
%


\subsection*{Acknowledgements}
Research partially supported by the State Research Agency (AEI) via the grants no.\ PID2020-117868GB-I00  and PID2023-150727NB-I00, and the ``Maria de Maeztu'' Unit of Excellence IMAG, reference CEX2020-001105-M, funded by MICIU/AEI/10.13039/501100011033 and ERDF/EU, Spain.

We are grateful to two anonymous referees for their helpful remarks and suggestions which led to improve the exposition.




\begin{thebibliography}{10}

\bibitem{abraham1963transversality}
R.~Abraham,
\newblock {\em Transversality in manifolds of mappings}, Bull. Amer. Math. Soc. \textbf{69} (1963), 470--474.

\bibitem{alarcon2018complete}
A.~Alarc{\'o}n and I.~Castro-Infantes,
\newblock {\em Complete minimal surfaces densely lying in arbitrary domains of {{\(\mathbb{R}^n\)}}}, Geom. Topol. \textbf{22} (2018), 571--590.

\bibitem{alarcon2019interpolation}
A.~Alarc{\'o}n and I.~Castro-Infantes,
\newblock {\em Interpolation by conformal minimal surfaces and directed holomorphic curves}, Anal. PDE \textbf{12} (2019), 561--604.

\bibitem{alarcon2014null}
A.~Alarc{\'o}n and F.~Forstneri{\v{c}},
\newblock {\em Null curves and directed immersions of open {Riemann} surfaces}, Invent. Math. \textbf{196} (2014), 733--771.

\bibitem{alarcon2025oka1}
A.~Alarc{\'o}n and F.~Forstneri{\v{c}},
\newblock {\em Oka-1 manifolds}, Math. Z. \textbf{311} (2025), Paper No. 33.

\bibitem{alarcon2016embedded}
A.~Alarc{\'o}n, F.~Forstneri{\v{c}}, and F.~J. L{\'o}pez,
\newblock {\em Embedded minimal surfaces in {{\(\mathbb R^n\)}}}, Math. Z. \textbf{283} (2016), 1--24.

\bibitem{alarcon2021minimal}
A.~Alarc{\'o}n, F.~Forstneri{\v{c}}, and F.~J. L{\'o}pez,
\newblock ``Minimal surfaces from a complex analytic viewpoint'', Springer Monogr. Math., Springer, Cham, 2021.

\bibitem{AlarconLarusson2023}
A.~Alarc{\'o}n and F.~L{\'a}russon,
\newblock {\em Regular immersions directed by algebraically elliptic cones}, J. Reine Angew. Math \textbf{823} (2025), 113--135.

\bibitem{alarcon2022algebraic} 
A.~Alarc{\'o}n and F.~J.~L{\'o}pez,
\newblock {\em Algebraic approximation and the {Mittag-Leffler} theorem for minimal surfaces}, Anal. PDE \textbf{15} (2022), 859--890.

\bibitem{alarcon2024generic}
A.~Alarc{\'o}n and F.~J.~L{\'o}pez,
\newblock {\em Generic properties of minimal surfaces}, arXiv preprint (2024), \url{https://arxiv.org/abs/2412.11563}.

\bibitem{AlarconLopez2012JDG}
A.~Alarc{\'o}n and F.~J. L{\'o}pez,
\newblock {\em Minimal surfaces in {$\mathbb R^3$} properly projecting into {$\mathbb R^2$}}, J. Differential Geom. \textbf{90} (2012), 351--381.

\bibitem{barbosa1986minimal}
J.~L.~M.~Barbosa and A.~G.~Corales,
\newblock ``Minimal surfaces in {{\({\mathbb{R}}^ 3\)}}'', volume~1195 of Lect. Notes Math.,
\newblock Springer-Verlag, Berlin, 1986.

\bibitem{behnke1949entwicklung}
H.~Behnke and K.~Stein,
\newblock {\em Entwicklung analytischer {F}unktionen auf {R}iemannschen {F}l\"{a}chen}, Math. Ann. \textbf{120} (1949), 430--461.

\bibitem{bishop1958subalgebras}
E.~Bishop,
\newblock {\em Subalgebras of functions on a {R}iemann surface}, Pacific J. Math. \textbf{8} (1958), 29--50.

\bibitem{Castro-InfantesChenoweth2020}
I.~Castro-Infantes and B.~Chenoweth,
\newblock {\em Carleman approximation by conformal minimal immersions and directed holomorphic curves}, J. Math. Anal. Appl. \textbf{484} (2020), 123756.

\bibitem{ChernOsserman1967JAM}
S.~S. Chern and R.~Osserman,
\newblock {\em Complete minimal surfaces in euclidean {$n$}-space}, J. Anal. Math. \textbf{19} (1967), 15--34.

\bibitem{chirka1989complex}
E.~M.~Chirka,
\newblock ``Complex analytic sets'', volume~46 of Math. Appl., Sov. Ser., 
\newblock Kluwer Academic Publishers Group, Dordrecht, 1989. Translated from the Russian by R.~A.~M.~Hoksbergen.

\bibitem{colding2011course}
T.~H.~Colding and W.~P.~Minicozzi, II,
\newblock ``A course in minimal surfaces'', volume~121 of Grad. Stud. Math.,
\newblock American Mathematical Society (AMS), Providence, RI, 2011.

\bibitem{colding1999minimal}
T.~H.~Colding and W.~P.~Minicozzi, II,
\newblock ``Minimal surfaces'', volume~4 of Courant Lect. Notes Math.,
\newblock Courant Institute of Mathematical Sciences, New York, NY, 1999.

\bibitem{dierkes2010minimal}
U.~Dierkes, S.~Hildebrandt, and F.~Sauvigny,
\newblock ``Minimal surfaces. {With} assistance and contributions by {A}. {K{\"u}ster} and {R}. {Jakob}'', volume 339 of Grundlehren Math. Wiss.,
\newblock Springer, Dordrecht, 2nd edition, 2010.

\bibitem{florack1948regulare}
H.~Florack,
\newblock {\em Regul\"{a}re und meromorphe {F}unktionen auf nicht geschlossenen {R}iemannschen {F}l\"{a}chen}, Schr. Math. Inst. Univ. M\"{u}nster \textbf{1948} (1948), 34.

\bibitem{fornaess2020}
J.~E.~Forn\ae ss, F.~Forstneri{\v{c}}, and E.~F.~Wold,
\newblock {\em Holomorphic approximation: the legacy of Weierstrass, Runge, Oka-Weil, and Mergelyan}, In: ``Advancements in complex analysis. From theory to practice'', D. Breaz et al. (eds.),
\newblock Springer, Cham, 2020, 133--192.

\bibitem{forstneric2009oka}
F.~Forstneri{\v{c}},
\newblock {\em Oka manifolds}, C. R. Math. Acad. Sci. Paris \textbf{347} (2009), 1017--1020.

\bibitem{forstneric2017stein}
F.~Forstneri{\v{c}},
\newblock ``Stein manifolds and holomorphic mappings. {The} homotopy principle in complex analysis'', volume~56 of Ergeb. Math. Grenzgeb., 3. Folge,
\newblock Springer, 2nd edition, Cham, 2017.

\bibitem{forstneric2023recent}
F.~Forstneri{\v{c}},
\newblock {\em Recent developments on Oka manifolds}, Indag. Math. (N.S.) \textbf{34} (2023), 367--417.
\bibitem{ForstnericLarusson2025oka1}
F.~Forstneri\v{c} and F.~L{\'a}russon,
\newblock {\em Oka-1 manifolds: new examples and properties}, Math. Z. \textbf{309} (2025), Paper No. 26.

\bibitem{ForstnericLarusson2019CAG}
F.~Forstneri\v{c} and F.~L{\'a}russon,
\newblock {\em The parametric \(h\)-principle for minimal surfaces in \(\mathbb{R}^n\) and null curves in \(\mathbb{C}^n\)}, Commun. Anal. Geom. \textbf{27} (2019), 1--45.

\bibitem{ForstnericLarusson2011survey}
F.~Forstneri\v{c} and F.~L{\'a}russon,
\newblock {\em Survey of {Oka} theory}, New York J. Math. \textbf{17a} (2011), 11--38.

\bibitem{grauert1958analytische}
H.~Grauert,
\newblock {\em Analytische {Faserungen} {\"u}ber holomorph-vollst{\"a}ndigen {R{\"a}umen}}, Math. Ann. \textbf{135} (1958), 263--273.

\bibitem{grauert1957approximation}
H.~Grauert,
\newblock {\em Approximationss{\"a}tze f{\"u}r holomorphe {Funktionen} mit {Werten} in komplexen {R{\"a}umen}}, Math. Ann \textbf{133} (1957), 139--159.

\bibitem{gromov1989oka}
M.~Gromov,
\newblock {\em Oka's principle for holomorphic sections of elliptic bundles}, J. Amer. Math. Soc. \textbf{2} (1989), 851--897.

\bibitem{gunning1967immersion}
R.~C.~Gunning and R.~Narasimhan,
\newblock {\em Immersion of open {Riemann} surfaces}, Math. Ann. \textbf{174} (1967), 103--108.

\bibitem{hoffman1997complete}
D.~Hoffman and H.~Karcher,
\newblock {\em Complete embedded minimal surfaces of finite total curvature}, In: ``Geometry V'', volume 90 of Encyclopaedia~Math.~Sci., R.~Osserman et al. (eds.),
\newblock Springer, Berlin, 1997, 5--93.

\bibitem{huber1957on}
A.~Huber,
\newblock {\em On subharmonic functions and differential geometry in the large}, Comment. Math. Helv. \textbf{32} (1957), 13--72.

\bibitem{JorgeMeeks1983T}
L.~P. d.~M. Jorge and W.~H.~Meeks,~III,
\newblock {\em The topology of complete minimal surfaces of finite total {G}aussian curvature}, Topology \textbf{22} (1983), 203--221.

\bibitem{lopez2014exotic}
F.~J.~L{\'o}pez,
\newblock {\em Exotic minimal surfaces}, J. Geom. Anal. \textbf{24} (2014), 988--1006.

\bibitem{lopez2014uniform}
F.~J.~L{\'o}pez,
\newblock {\em Uniform approximation by complete minimal surfaces of finite total curvature in {{\(\mathbb{R}^{3}\)}}}, Trans. Amer. Math. Soc. \textbf{366} (2014), 6201--6227.

\bibitem{LopezMartin1999}
F.~J. L{\'o}pez and F.~Mart\'{i}n,
\newblock {\em Complete minimal surfaces in {$\mathbf R^3$}}, Publ. Mat. \textbf{43} (1999), 341--449.

\bibitem{meeks2011classical}
W.~H.~Meeks,~III and J.~P{\'e}rez,
\newblock {\em The classical theory of minimal surfaces}, Bull. Amer. Math. Soc. (N.S.) \textbf{48} (2011), 325--407.

\bibitem{meeks2012survey}
W.~H.~Meeks,~III and J.~P{\'e}rez,
\newblock ``A survey on classical minimal surface theory'', volume 60 of Univ. Lecture Ser.,
\newblock American Mathematical Society (AMS), Providence, RI, 2012.

\bibitem{mittag1884demonstration}
G.~Mittag-Leffler,
\newblock {\em D\'{e}monstration nouvelle du th\'{e}or\`eme de {L}aurent}, Acta Math. \textbf{4} (1884), 80--88.

\bibitem{morales2003existence}
S.~Morales,
\newblock {\em On the existence of a proper minimal surface in \(\mathbb{R}^3\) with a conformal type of disk}, Geom. Funct. Anal. \textbf{13} (2003), 1281--1301.

\bibitem{nitsche1989lectures}
J.~C.~C.~Nitsche,
\newblock ``Lectures on minimal surfaces'', Cambridge University Press, Cambridge, 1989.

\bibitem{osserman1986survey}
R.~Osserman,
\newblock ``A survey of minimal surfaces'', Dover Publications, Inc., 2nd edition, New York, 1986.

\bibitem{pirola1998algebraic}
G.~P.~Pirola,
\newblock {\em Algebraic curves and non-rigid minimal surfaces in the {Euclidean} space}, Pacific J. Math. \textbf{183} (1998), 333--357.

\bibitem{SchoenYau1997lectures}
R.~Schoen and S.-T.~Yau,
\newblock ``Lectures on harmonic maps'', volume 2 of Conf. Proc. Lect. Notes Geom. Topol., 
\newblock International Press, Cambridge, MA, 1997.

\bibitem{yang1994complete}
K.~Yang,
\newblock ``Complete minimal surfaces of finite total curvature'', volume 294 of Math. Appl.,
\newblock Kluwer Academic Publishers, Dordrecht, 1994.

\end{thebibliography}


\medskip
\noindent Antonio Alarc\'{o}n, Tja\v{s}a Vrhovnik

\smallskip
\noindent Departamento de Geometr\'{\i}a y Topolog\'{\i}a e Instituto de Matem\'aticas (IMAG), Universidad de Granada, Campus de Fuentenueva s/n, E--18071 Granada, Spain.

\smallskip
\noindent  e-mail: {\tt alarcon@ugr.es}, {\tt vrhovnik@ugr.es}

\end{document}